\documentclass[11pt, reqno,amscd]{amsart}
\usepackage{fullpage}
\usepackage{amsthm, amsmath}
\usepackage{amssymb}
\usepackage{amscd}
\usepackage{graphicx}
\usepackage{amsfonts}
\usepackage{color}

\usepackage[all]{xy}
\usepackage{rotating, graphicx, blindtext}

\theoremstyle{plain}
\newtheorem{theorem}{Theorem}[section]

\newtheorem{corollary}[theorem]{Corollary}

\newtheorem{lemma}[theorem]{Lemma}

\theoremstyle{remark}
\newtheorem{remark}[theorem]{Remark}

\theoremstyle{definition}
\newtheorem{definition}[theorem]{Definition}

\newcommand{\schememu}{\mu\!\!\!\mu_n}

\begin{document}

\title[Cup products in the \'etale cohomology of number fields]
{Cup products in the \'etale cohomology of number fields}

\author[\scriptsize  Bleher]{\scriptsize F. M. Bleher}
\address{F. M. Bleher, Dept. of Mathematics\\Univ. of Iowa\\Iowa City, IA 52242, USA}
\email{frauke-bleher@uiowa.edu}
\thanks{F. B. was partially supported by NSF FRG Grant No.\ DMS-1360621.}

\author[Chinburg]{T. Chinburg}
\address{T. Chinburg, Dept. of Mathematics\\ Univ. of Pennsylvania \\ Philadelphia, PA 19104, USA}
\email{ted@math.upenn.edu}
\thanks{T. C. was partially supported by  NSF FRG Grant No.\ DMS-1360767, NSF FRG Grant No.\ DMS-1265290,
NSF SaTC grant No. CNS-1513671, Simons Foundation grant 338379 and NSF Grant No.\ DMS 1107452, 1107263, 1107367 "RNMS: Geometric Structures and Representation Varieties" (the GEAR Network)}

\author[Greenberg]{R. Greenberg}
\address{R. Greenberg\\Dept. of Mathematics\\Univ. of Washington\\Box 354350\\ Seattle, WA 98195, USA}
\email{greenber@math.washington.edu}
\thanks{R. G. was partially supported by NSF FRG Grant No.\ DMS-1360902}

\author[Kakde]{M. Kakde}
\address{M. Kakde\\Dept. of Mathematics\\King's College\\Strand\\London WC2R 2LS, UK}
\email{mahesh.kakde@kcl.ac.uk}

\author[Pappas]{G. Pappas}
\address{G. Pappas\\ Dept. of
Mathematics\\
Michigan State Univ.\\
E. Lansing, MI 48824, USA}
\email{pappas@math.msu.edu}
\thanks{G. P.  was partially supported by  NSF FRG Grant No.\ DMS-1360733.}

\author[Taylor]{M. J. Taylor}
\address{M. J. Taylor\\ Merton College, Univ. of  
Oxford\\ Oxford, OX1 4JD, UK}
\email{martin.taylor@merton.ox.ac.uk}



\date{\today}

\begin{abstract}
This paper concerns cup product pairings in \'etale cohomology related to work of M. Kim and of W. McCallum and R. Sharifi.  We will show that by considering Ext groups rather than cohomology groups, one arrives at a pairing which combines invariants defined by Kim with a pairing defined by McCallum and Sharifi.  We also prove a formula  for Kim's invariant in terms of Artin maps in the case of cyclic unramified Kummer extensions.  One consequence is that for all $n > 1$, there are infinitely many number fields $F$ over which there are both trivial and non-trivial Kim invariants associated to cyclic groups of order $n$.  \end{abstract}

\maketitle

\section{Introduction} 
\label{s:introsect}
\setcounter{equation}{0}

This paper concerns cup product pairings in \'etale cohomology which underlie an important case of
the arithmetic Chern-Simons theory introduced by M. Kim in \cite{Kim} as well as a pairing in Galois cohomology studied by
McCallum and Sharifi in \cite{McS}.  Our interest in these pairings arises from the search for  new numerical invariants 
of number fields which pertain to the higher codimension behavior of Iwasawa modules (see \cite{B}).

Suppose $F$ is a number field and $O_F$ is its ring of integers. Let $X = \mathrm{Spec}(O_F)$ and let
$\mu_n$ be the sheaf of $n^{th}$ roots of unity in the \'etale topology on $X$.
The pairing connected with Kim's work is the natural cup product pairing
\begin{equation}
\label{eq:pairit}
\xymatrix{  H^1(X,\mathbb{Z}/n) \times H^2(X,\mu_n) \ar[r] & H^3(X,\mu_n) \ar@{=}[r]^(.6){\mathrm{inv}_n} &\mathbb{Z}/n  }
\end{equation}
in \'etale cohomology when 
$\mathrm{inv}_n$ is the invariant map isomorphism (see  \cite[p. 538]{Mazur}).

Suppose $F$ contains the multiplicative group $\tilde{\mu}_n$ generated by a primitive $n^{th}$ root of unity, and
let $G$ be an abstract finite group  acting trivially on $\tilde{\mu}_n= \mu_n(X)$.
Let $\pi_1(X,\eta)$ be the \'etale fundamental group 
of $X$ relative to a fixed base point $\eta$. Then $\pi_1(X,\eta)$ is the Galois group of a maximal everywhere unramified extension of $F$.  
Suppose $c$ is a class in $H^3(G,\tilde{\mu}_n)$, and let $f:\pi_1(X,\eta) \to G$ be a fixed homomorphism. 
Then $f^*c \in H^3(\pi_1(X,\eta),\tilde{\mu}_n)$ defines via \v{C}ech cohomology a class
$f^*_X c \in H^3(X,\mu_n)$. Kim's invariant in \cite{Kim} in the unramified case is
\begin{equation}
\label{eq:kimclass}
S(f,c) = \mathrm{inv}_n(f^*_X c) \in \mathbb{Z}/n.
\end{equation}
In the ramified case, one replaces $X$ by the complement $X'$ of a non-empty finite set of closed points of $X$. 
One must then take a different approach, since $H^3(X',\mu_n) = \{0\}$;  see \cite{Kim}.  We will return to the ramified
case in a later paper.  

One way to compute (\ref{eq:kimclass}) is to employ the pairing (\ref{eq:pairit}). Namely, 
consider the
diagram of pairings 
\begin{equation}
\label{eq:cuppairnew}
\xymatrix @C=.2pc{
H^1(G,\mathbb{Z}/n)\ar[d]^{f^*_X} &\times &H^2(G,\tilde{\mu}_n)\ar[d]^{f^*_X} &\ar[rrr]&&&& H^3(G,\tilde{\mu}_n)\ar[d]^{f^*_X}\\
 H^1(X,\mathbb{Z}/n) &\times &H^2(X,\mu_n) &\ar[rrr]&&&& H^3(X,\mu_n)} 
\end{equation}
in which the vertical homomorphisms are induced by $f$.  Picking classes $c_1 \in H^1(G,\mathbb{Z}/n)$
and $c_2 \in H^2(G,\tilde{\mu}_n)$ such that $c_1 \cup c_2 = c$, the pairing (\ref{eq:pairit}) leads to a way to compute
\begin{equation}
\label{eq:Kimdef}
S(f,c) = f^*_X(c_1) \cup f_X^*(c_2).
\end{equation}
The McCallum-Sharifi pairing, on the other hand, is defined using Galois cohomology. It was 
defined in \cite{McS} using the cup product pairing
 \begin{equation}
 \label{eq:MC1}
 H^1(G_{F,S},\tilde{\mu}_n) \times H^1(G_{F,S},\tilde{\mu}_n) \to H^2(G_{F,S},\tilde{\mu}_n^{\otimes 2})
 \end{equation}
when $S$ is a finite set of places of $F$ containing all the places above $n$ and all real archimedean places, and $G_{F,S}$ is the Galois group of the 
maximal unramified outside $S$ extension of $F$.  

A pairing which incorporates both Kim's invariant for $G = \mathbb{Z}/n$ and the 
McCallum-Sharifi pairing is the cup product Ext pairing
\begin{equation}
\label{eq:mcsk}
\mathrm{Ext}^1_X(\mathbb{Z}/n,\mu_n) \times \mathrm{Ext}^2_X(\mathbb{Z}/n,\mu_n) \to 
\mathrm{Ext}^3_X(\mathbb{Z}/n,\mu_n^{\otimes 2}).
\end{equation}
To explain this, consider the exact sequence 
$$0 \to \mathbb{Z} \xrightarrow{\cdot n} \mathbb{Z} \to \mathbb{Z}/n \to 0$$
induced by multiplication by $n$.  The long exact Ext sequence associated to this sequence leads to a diagram 
\begin{equation}
\label{eq:bowow}
\xymatrix @C=.2pc{
0\ar[d]&&0\ar[d]&&&&&0\ar[d]\\
H^0(X,\mu_n)\ar[d]& &H^1(X,\mu_n)\ar[d] &&&&& H^2(X,\mu_n^{\otimes 2})\ar[d]  \\
 \mathrm{Ext}^1_X(\mathbb{Z}/n,\mu_n)\ar[d] & \times&\mathrm{Ext}^2_X(\mathbb{Z}/n,\mu_n)\ar[d]&\ar[rrr] &&&& \mathrm{Ext}^3_X(\mathbb{Z}/n,\mu_n^{\otimes 2})\ar[d]\\
 H^1(X,\mu_n) \ar[d]&\times &H^2(X,\mu_n)\ar[d] &\ar[rrr]&&&& H^3(X,\mu_n^{\otimes 2})\ar[d]\\
 0&&0&&&&&0
}
\end{equation}
in which the vertical sequences are exact and the pairings in the second and third rows are given by cup products.  Note
that we have natural isomorphisms 
\begin{equation}
\label{eq:tensoriso}
H^i(X,\mu_n^{\otimes j}) = H^i(X,\mathbb{Z}/n) \otimes \tilde{\mu}_n^{\otimes j}
\end{equation}
for all $i, j \ge 0$ since $\tilde{\mu}_n = H^0(X,\mu_n)$ has order $n$ by assumption.  

We show  the following result in \S \ref{s:prooflast}.  

\begin{theorem}
\label{thm:bigdiagram}The cup product in the bottom row of (\ref{eq:bowow}) can be used to compute Kim's invariant via (\ref{eq:cuppairnew}), (\ref{eq:Kimdef}) and (\ref{eq:tensoriso}).  This pairing is compatible with 
pushing forward the cup product in the middle row of (\ref{eq:bowow}).
The cup product pairing
\begin{equation}
\label{eq:natural}
H^1(X,\mu_n) \times H^1(X,\mu_n) \to H^2(X,\mu_n^{\otimes 2})
\end{equation}
is compatible with the McCallum-Sharifi pairing, which results from (\ref{eq:MC1}), via the natural inflation maps $H^i(X,\mu_n) \to H^i(G_{F,S},\tilde{\mu}_n)$.  
The pairing (\ref{eq:natural})
arises from the pairing in the second row of (\ref{eq:bowow}) by the natural pull back and push forward procedure.  Namely, suppose $\alpha \in H^1(X,\mu_n) $ pulls back to $\tilde{\alpha} \in \mathrm{Ext}^1_X(\mathbb{Z}/n,\mu_n)$ in the first column of (\ref{eq:bowow}),
and that $\beta \in H^1(X,\mu_n)$ has boundary $\partial \beta \in \mathrm{Ext}^2_X(\mathbb{Z}/n,\mu_n)$ under the first vertical map in
the second column of (\ref{eq:bowow}).  
Then 
\begin{equation}
\label{eq:cup}
\partial(\alpha \cup \beta) = -(\tilde{\alpha} \cup \partial \beta)
\end{equation}
where on the left $\partial$ is the boundary map
$H^2(X,\mu_n^{\otimes 2}) \to \mathrm{Ext}^3_X(\mathbb{Z}/n,\mu_n^{\otimes 2})$
in the third column of (\ref{eq:bowow}). 
 \end{theorem}
 
 Note that the minus sign on the right side of (\ref{eq:cup}) comes from the definition of the differential of the total complex
of the tensor product of two complexes.  

Another pairing in Galois cohomology that is related to Kim's invariants and
different from the McCallum-Sharifi pairing is described in Theorem \ref{thm:fixitup} below.

In \cite{KimEtAl}, H. Chung, D. Kim, M. Kim, J. Park and H. Yoo showed how to compute Kim's invariant  by comparing local
and global trivializations of Galois three cocycles.  Using this method
they construct infinitely many examples in which the invariant is non-trivial and the finite group involved is either $\mathbb{Z}/2$, $\mathbb{Z}/2 \times \mathbb{Z}/2$ or the
symmetric group $S_4$.

Our next results use a different approach than \cite{KimEtAl} in the unramified case.   When $G$ is cyclic we prove in Theorem \ref{thm:fixed?} below a formula  that determines the invariant using Artin maps.   
One consequence of Theorem \ref{thm:fixed?}  is  the following result. This  shows that
there are infinitely many number fields $F$ over which there are both trivial and non-trivial Kim invariants associated to cyclic groups of order $n$. The methods of this paper carry over \textit{mutatis mutandis} to the case of global function fields provided $n$ is prime to the characteristic of the field.

\begin{theorem}
\label{thm:result} 
Suppose $n > 1$ is an integer, $G = \mathbb{Z}/n$ and that $c$ is a fixed generator of $H^3(G,\tilde{\mu}_n)$.  Then there are infinitely many totally complex number fields $F$ for which there are cyclic unramified Kummer extensions $K_1/F$ and $K_2/F$ with the following property. Let $f_i:\pi_1(X,\eta) \to G$ for $i = 1, 2$ be the inflation of an isomorphism $\mathrm{Gal}(K_i/F) \to G$.  Then 
\begin{equation}
\label{eq:Kiminvariant}
S(f_1,c) = 0 \quad \mathrm{and}\quad  S(f_2,c) \ne 0.
\end{equation}
\end{theorem}

To state our formula for Kim's invariant in terms of Artin maps, let  $f:\pi_1(X,\eta) \to G = \mathbb{Z}/n$ be a fixed surjection.  Let $c_1 \in H^1(G,\mathbb{Z}/n) = \mathrm{Hom}(G,\mathbb{Z}/n)$ 
be the identity map, and let $c_2$ generate $H^2(G,\tilde{\mu}_n)$. Then $c = c_1 \cup c_2$ generates
the cyclic group $H^3(G,\tilde{\mu}_n)$ of order $n$. We wish to use the diagram (\ref{eq:cuppairnew}) to calculate $S(f,c)=f_X^*(c_1) \cup f_X^*(c_2)$.

The element $f_X^*(c_1) \in H^1(X,\mathbb{Z}/n)$ factors through an isomorphism 
$$\mathrm{Gal}(K/F) \to G = \mathbb{Z}/n$$ for a cyclic unramified extension $K/F$ of degree $n$ 
which we will use to identify $\mathrm{Gal}(K/F)$ with $G = \mathbb{Z}/n$.  

Using the exact sequence of multiplicative groups
\begin{equation}
\label{eq:exactly}
1 \to \tilde{\mu}_n \to K^* \to K^* \to K^*/(K^*)^n \to 1
\end{equation}
associated to exponentiation by $n$ on $K^*$ we will show that there is an exact sequence
\begin{equation}
\label{eq:boundary}
F^* \to (K^*/(K^*)^n)^{\mathrm{Gal}(K/F)} \to H^2(\mathrm{Gal}(K/F),\tilde{\mu}_n) \to 1.
\end{equation}
Let $\gamma \in K^*$ be such that $\gamma (K^*)^n \in (K^*/(K^*)^n)^G$ has image $c_2$ in 
$$H^2(\mathrm{Gal}(K/F),\tilde{\mu}_n) = H^2(G,\tilde{\mu}_n)$$ 
under the homomorphism in (\ref{eq:boundary}). 

\begin{theorem}
\label{thm:fixed?}
 The $O_F$ ideal $\mathrm{Norm}_{K/F}(\gamma) O_F$ is the $n^{th}$ power $I^n$ of
a fractional ideal $I$ of $F$.  The ideal class $ [I]$ of $I$ in the ideal class group $Cl(O_F)$ of $O_F$  depends only on $c_2$ and is $n$-torsion.  Let $\mathrm{Art}:Cl(O_F) \to \mathrm{Gal}(K/F) = G = \mathbb{Z}/n$ be the Artin map associated to $K/F$.  Then Kim's invariant of the class 
$c = c_1 \cup c_2 \in H^3(G,\tilde{\mu}_n)$ is 
\begin{equation} 
\label{eq:calculation}
S(f,c) = f_X^* c_1 \cup f_X^*c_2 = \mathrm{Art}([I]) \in G = \mathbb{Z}/n.
\end{equation}
\end{theorem}

Note that in this result, the input is $f$ and $c_2$, from which one determines $K$ and $\gamma$.  Conversely,
we now show how one can start with a cyclic unramified degree $n$ Kummer extension $K/F$ and then use this to determine
 an $f$ and $c_2$ for which (\ref{eq:calculation}) holds.   
 
For the remainder of the paper we fix the following choices.
 
\begin{definition}
\label{def:roots} Let $\zeta_n$ be a  primitive $n^{th}$ root of unity in $F$. If $m$ is a divisor of $n$, we let $\zeta_m = \zeta_n^{n/m}$.
\end{definition}

\begin{theorem}
\label{thm:elementary}
Suppose $K/F$ is an everywhere unramified cyclic degree $n$ Kummer extension of number fields.    By Hilbert's norm theorem, $\zeta_n = \mathrm{Norm}_{K/F}(x)$ for some $x \in K$.  By Hilbert's Theorem 90, $x^n = \sigma(y)/y$ for some $y \in K$ and a generator $\sigma$ for $G = \mathrm{Gal}(K/F)$.  For all such $y$, there is a fractional $O_F$-ideal $J$ such that $\mathrm{Norm}_{K/F}(y)O_F = J^n$. Let $c_1:G \to \mathbb{Z}/n$ be the isomorphism sending $\sigma$ to $1$ mod $n$.  Let $\gamma = y$ in   Theorem \ref{thm:fixed?}, and let $c_2 \in H^2(G,\tilde{\mu}_n)$ be the image of $\gamma (K^*)^n$ under (\ref{eq:boundary}).  Then $c_2$ generates 
$H^2(G,\tilde{\mu}_n)$, $J$ is the ideal $I$ of Theorem \ref{thm:fixed?} and
$S(f,c)$ is given by  (\ref{eq:calculation}) when $c = c_1 \cup c_2$.  
\end{theorem}

This theorem leads to the following result concerning the functorality of Kim's invariant under
base extensions.

\begin{corollary}  
\label{cor:elemcor1}
Suppose $F'$ is a finite extension of $F$ which is disjoint from $K$, and let
$K' = F'K$ be the compositum of $F'$ and $K$.  The ideal $I'$ associated to $K'/F'$
by Theorem \ref{thm:elementary} may be taken to be $IO_{F'}$.  Kim's invariant for 
$K'/F'$ is the image of the invariant for $K/F$ under the transfer map $Tr_{K'/K}:\mathrm{Gal}(K/F) \to \mathrm{Gal}(K'/F')$ we identify both of these Galois groups with $\mathbb{Z}/n$.  
\end{corollary}

Theorem \ref{thm:elementary}  gives the following criterion for the non-triviality of Kim's invariant for cyclic unramified Kummer extensions.

\begin{corollary}
\label{cor:elemcor2}
With the notations of Theorem \ref{thm:elementary}, the following are equivalent:
\begin{enumerate}
\item[i.]  The invariant $S(f,c)$ is trivial for all 
$f:\pi_1(X,\eta) \to G = \mathbb{Z}/n$ factoring through $\mathrm{Gal}(K/F)$ and all $c \in H^3(G,\tilde{\mu}_n)$.
\item[ii.] 
$[J]$ is contained in $\mathrm{Norm}_{K/F}(Cl(O_K))$.
\item[iii.] The image of $[J]$ under the Artin map 
$\mathrm{Art}:Cl(O_F) \to 
\mathrm{Gal}(K/F)$ is trivial. 
\end{enumerate}
\end{corollary}

We now describe another way to find an element $\gamma \in K$ with the properties
in Theorem \ref{thm:fixed?}.  This method will be used to show Theorem \ref{thm:result}.

\begin{theorem}
\label{thm:intrinsic}  Let $f:\pi_1(X,\eta) \to G = \mathbb{Z}/n = \mathrm{Gal}(K/F)$ be as above with $c_1:G \to \mathbb{Z}/n$ the
identity map.
\begin{enumerate}
\item[i.] There is a cyclic degree $n^2$ extension $L/F$ such that $K \subset L$.
This extension is unique up to twisting by a cyclic degree $n$ extension of $F$, in the following sense.  Write $L = K(\gamma^{1/n})$ for some Kummer generator $\gamma$.   If $L'$ is any other cyclic degree $n^2$ extension of $F$ which contains $K$, then $L' = K(\gamma'^{1/n})$ for some  $\gamma' \in \gamma \cdot (K^*)^n \cdot F^*$, and conversely all such $\gamma'$ give rise to such $L'$.  
\item[ii.] The coset $\gamma (K^*)^n$ of $K^*/(K^*)^n$ is fixed by the action of $\mathrm{Gal}(K/F)$.  Let $c_2 \in H^2(\mathrm{Gal}(K/F),\tilde{\mu}_n) = H^2(G,\tilde{\mu}_n)$ be the image of $\gamma (K^*)^n$ under the boundary map in (\ref{eq:boundary}).  
The formula in (\ref{eq:calculation})
determines $S(f,c)$ when $c$ is the generator $c_1 \cup c_2$ of $H^3(G,\mu_n)$.  
\item[iii.]Suppose the ideal $I$
of Theorem \ref{thm:fixed?} has the form $I = I' \cdot J'$ for some fractional ideals $I'$ and $J'$ of $O_F$ such that any prime in the support of $J'$ is either split in $K$ or unramified in $L/K$.  Then 
\begin{equation}
\label{eq:neater}
S(f,c) = f_X^* c_1 \cup f_X^*c_2  = \mathrm{Art}([I']) \in G = \mathbb{Z}/n.
\end{equation}
\end{enumerate}
\end{theorem}

This description leads to the following corollaries, which we will show lead to a proof of Theorem \ref{thm:result}.

\begin{corollary}
\label{cor:easy}  Suppose $K/F$ is contained in a cyclic degree $n^2$ extension $L/F$ 
such that every prime $\mathcal{P}$ of $O_F$ which ramifies in $L$ splits completely in $K$. Then 
$S(f,c) = 0$ for all surjections $f:\pi_1(X,\eta) \to \mathrm{Gal}(K/F)  = G = \mathbb{Z}/n$
and all $c \in H^3(G,\tilde{\mu}_n)$.
\end{corollary}

\begin{corollary}
\label{cor:almost} Suppose $K/F$ is contained in a cyclic degree $n^2$ extension $L/F$
with the following properties.  There is a unique prime ideal $\mathcal{P}$ of $O_F$ which ramifies in $L/F$, $\mathcal{P}$ is undecomposed in $L$ and the inertia group of $\mathcal{P}$ in $\mathrm{Gal}(L/F)$ is $\mathrm{Gal}(L/K)$. Furthermore, the residue characteristic of $\mathcal{P}$ is prime to $n$.  Then $S(f,c)$ is of order $n$ for all surjections $f:\pi_1(X,\eta) \to \mathrm{Gal}(K/F) = G = \mathbb{Z}/n$
and all generators $c$ of $H^3(G,\tilde{\mu}_n)$.
\end{corollary}

\begin{remark}
\label{rem:niceremark} These corollaries explain the examples of \cite[\S 5.5]{KimEtAl} in the following way.
Let $n = 2$, and let $F = \mathbb{Q}(\sqrt{-pt})$ where $p$ is a prime such that $p \equiv 1$ mod $4$ and
$t$ is a positive square-free integer prime to $p$.  Let $K$ be $F(\sqrt{p})$. Then $K$ is contained in the unique cyclic
degree $4$ extension $L$ of $F$ contained in $F(\tilde{\mu}_p)$.   The unique prime $\mathcal{P}$ over $p$ in $F$ is the unique prime which ramifies in $L$.  The examples in \cite[\S 5.5]{KimEtAl} arise from Corollaries \ref{cor:easy} and \ref{cor:almost} because  
$\mathcal{P}$ splits in $K$ if and only 
    if $t$ is a square mod $p$ since $-1$ is a square mod $p$. 
\end{remark}

The following two results give examples in which our results show that Kim's invariants are trivial, where $\zeta_n\in F$ is fixed as in Definition \ref{def:roots}.

\begin{theorem}
\label{thm:Ralph1}  Suppose $n$ is a properly irregular prime in the sense that $n$ divides $\# Cl(\mathbb{Z}[\zeta_n])$
but not $\# Cl(\mathbb{Z}[\zeta_n + \zeta_n^{-1}])$. If $K$ is any cyclic unramified extension of $F  = \mathbb{Q}(\zeta_n)$
then $S(f,c) = 0$ for all surjections $f:\pi_1(X,\eta) \to \mathrm{Gal}(K/F)  = G = \mathbb{Z}/n$
and all $c \in H^3(G,\tilde{\mu}_n)$.
\end{theorem}

\begin{theorem}
\label{thm:Ralph2}  
 Suppose that $n > 2$ is prime and $K/F$ is a cyclic unramified Kummer extension of degree $n$ such that both $K$ and $F$ are Galois over $\mathbb{Q}$.  Then $S(f,c) = 0$ for all surjections $f:\pi_1(X,\eta) \to \mathrm{Gal}(K/F)  = G = \mathbb{Z}/n$
and all $c \in H^3(G,\tilde{\mu}_n)$.  
\end{theorem}

\begin{remark}
\label{rem:notsoniceremark} Note that in Theorem \ref{thm:Ralph1},  $n$ does not divide $[F:\mathbb{Q}]$.  One can also construct many examples of Theorem \ref{thm:Ralph2} in which $[F:\mathbb{Q}]$ is prime to $n$.  However, the examples we will construct in Theorem \ref{thm:result} in which Kim's invariant is non-trivial all have $n|[F:\mathbb{Q}]$.  It would be interesting to find  examples in which Kim's invariant
is non-trivial when $n$ is prime and $[F:\mathbb{Q}]$ is not divisible  by $n$. 
\end{remark}

We now describe a  pairing in Galois cohomology that is different from the McCallum-Sharifi pairing and that gives rise to Kim's invariants.

Define 
\begin{equation}
\label{eq:tdef}
T(F) = \{a \in F^*:  F(a^{1/n})/F \quad \mathrm{is \ unramified}\}.
\end{equation}
Suppose $a \in T(F)$ and $b \in O_F^*$.  The field $K = F(a^{1/n})$ is a cyclic Kummer extension of degree $m$ dividing $n$.
Since $K/F$ is unramified, and $b \in O_F^*$, we have $b = \mathrm{Norm}_{K/F}(x)$ for some $x \in K^*$. 
Let $\sigma \in  \mathrm{Gal}(K/F)$ be the unique generator such that $\sigma(a^{1/n})/a^{1/n} = \zeta_m = \zeta_n^{n/m}$, where $\zeta_n\in F$ is as in Definition \ref{def:roots}.  
Since $\mathrm{Norm}_{K/F}(x^m/b) = 1$, there is an element $\nu \in K^*$ such that
$x^m/b = \sigma(\nu)/\nu$. Since $K/F$ is unramified and $b \in O_F^*$, the ideal $\mathrm{Norm}_{K/F}(\nu a^{1/n}) O_F$ equals $I^m$ for some fractional ideal $I$ of $O_F$.  We define
\begin{equation}
\label{eq:defined}
(a,b)_n = [I] \otimes \zeta_m \in Cl(O_F) \otimes_\mathbb{Z} \tilde{\mu}_n
\end{equation}
where $[I]$ is the ideal class of $I$ in $Cl(O_F)$.  The value $(a,b)_n$ does not depend on the choice of $\zeta_n$ 
in Definition \ref{def:roots}.

\begin{theorem} 
\label{thm:fixitup}
Suppose $K = F(a^{1/n})$ has degree $n  = m$ over $F$ for some $a \in T(F)$.  Let  $\sigma \in G$ be the generator such that 
$\sigma(a^{1/n})/a^{1/n} = \zeta_n$ and set $b = \zeta_n$.
Fix an isomorphism $c_1: G = \mathrm{Gal}(K/F) \to \mathbb{Z}/n$ by letting $\sigma \in G$ correspond to $1 \in  \mathbb{Z}/n$. 
Then as above, $b = \mathrm{Norm}_{K/F}(x)$ for some $x \in K^*$ and $x^n/b = x^n/\zeta_n = \sigma(\nu)/\nu $ for some  $\nu \in K^*$. When $y = \nu a^{1/n}$, the coset $y(K^*)^n$ lies in $(K^*/(K^*)^n)^G$, and its image under the homomorphism in (\ref{eq:boundary})
is a generator $c_2 \in H^2(G,\tilde{\mu}_n)$.  Let $c = c_1 \cup c_2 \in H^3(G,\tilde{\mu}_n)$. There is a unique homomorphism $\kappa: Cl(O_F) \otimes_\mathbb{Z} \tilde{\mu}_n \to \mathbb{Z}/n
$ which 
sends $[J]\otimes \zeta_n$ to $\mathrm{Art}([J]) \in G = \mathbb{Z}/n$ for all fractional ideals $J$ of $O_F$.  
Kim's invariant $S(f,c)$ is given by 
\begin{equation}
\label{eq:Sfcor}
S(f,c) = \kappa( (a,\zeta_n)_n)
\end{equation}
when $(a,\zeta_n)_n$ is the pairing defined by (\ref{eq:defined}). 
\end{theorem}

By contrast, the McCallum-Sharifi pairing is defined in the following way.  Let $S$ be the union of set of places of $F$ which have residue characteristics dividing $n$ with the real places.  Let $C_{F,S}$ be the $S$-class group of $F$.  
In \cite[\S 2]{McS} McCallum and Sharifi define a pairing 
\begin{equation}
\label{eq:MS}
\langle \ , \ \rangle_S : T(F) \times O_F^* \to (C_{F,S}/nC_{F,S}) \otimes \tilde{\mu}_n.
\end{equation}
See also \cite{Shar} for further discussion.

\begin{remark}
\label{ex:Kimex} 
 Here is an example for which the following three statements hold:
\begin{enumerate}
\item[i.] Kim's invariant $S(f,c)$ in (\ref{eq:Sfcor}) is not trivial.
\item[ii.]  The McCallum-Sharifi pairing value $\langle a,\zeta_n\rangle_S$ in (\ref{eq:MS}) is trivial.
\item[iii.]  The homomorphism $Cl(O_F) \to C_{F,S}$ is an isomorphism.
\end{enumerate}
Let $n = 2$, $G = \mathbb{Z}/2$ and $\zeta_n = -1$. Define $F = \mathrm{Q}(\sqrt{-p t})$ for some prime $p \equiv 1$ mod 
$4$ and some square-free $t > 0$ such that $t$ is not a square mod $p$ and $-pt \equiv 5$ mod $8$.  When $K = F(\sqrt{p})$ and $a = p$, Remark \ref{rem:niceremark} shows (i).  Since $2$ is inert to $F$, (iii) holds when $S$ is the set of places of $F$ over $2$.
Finally (ii) follows from the formula in \cite[Thm. 2.4]{McS} since $\zeta_2 = -1 = \mathrm{Norm}_{K/F}(\epsilon) $ when 
$\epsilon$ is a fundamental unit of $\mathbb{Q}(\sqrt{p})\subset K$.
\end{remark}

 \medbreak
 
\noindent {\bf Acknowledgements.} 
We would like to thank the authors of \cite{KimEtAl} for sending us a preprint of their work, which led to our correcting some errors in an earlier version of this paper.
We would also like to thank Romyar Sharifi and Roland van der Veen for many very helpful conversations and suggestions about this work.   After this paper was written, Theorem \ref{thm:fixed?} as well as other pairings related to Kim's invariant have been  investigated further in \cite{KimEtAl2}.  The authors would like to thank the referee for very helpful comments. 

\medbreak

\section{Proof of Theorem \ref{thm:bigdiagram}}
\label{s:prooflast}
\setcounter{equation}{0}

We assume in this section the notations of Theorem \ref{thm:bigdiagram}.  We will
use the results of Swan in \cite{Swan} concerning cup products.  In \cite[\S3]{Swan}, Swan considers cup products of covariant left exact functors.
This can be used to define the cup product pairings in the middle and bottom rows of (\ref{eq:bowow}).
Namely, in the category of sheaves in the \'etale  topology on $X$, let
$$0 \to U \to I^0 \to I^1 \to I^2\to I^3\to \cdots$$
and 
$$0 \to V \to J^0 \to J^1 \to J^2\to J^3\to \cdots$$
be pure injective resolutions. Then the total complex $I^\bullet\otimes J^\bullet$ is a pure, but not necessarily injective, resolution of $U\otimes V$. Let
$$0 \to U\otimes V\to K^0 \to K^1 \to K^2\to K^3\to \cdots$$
be a pure injective resolution, and choose a morphism of resolutions $I^\bullet\otimes J^\bullet\to K^\bullet$ over $U\otimes V$. 
For \'etale sheaves $A,B$ on $X$, the composition of morphisms 
\begin{equation}
\label{eq:cupproductmaps}
\mathrm{Hom}_X(A,I^\bullet)\otimes \mathrm{Hom}_X(B,J^\bullet)\to \mathrm{Hom}_X(A\otimes B,I^\bullet\otimes J^\bullet) \to \mathrm{Hom}_X(A\otimes B,K^\bullet)
\end{equation}
then induces a cup product pairing
$$\mathrm{Ext}^i_X(A,U)\times \mathrm{Ext}^j_X(B,V) \to \mathrm{Ext}^{i+j}_X(A\otimes B, U\otimes V)$$
(see \cite[Thm. 3.4, Lemma 3.6 and \S7]{Swan}). Given morphisms of \'etale sheaves $C\to A\otimes B$ and $U\otimes V\to T$, we get
\begin{equation}
\label{eq:cupwithpureinjective}
\mathrm{Ext}^i_X(A,U)\times \mathrm{Ext}^j_X(B,V) \to \mathrm{Ext}^{i+j}_X(C,T).
\end{equation}

The first statement in Theorem \ref{thm:bigdiagram} is that the cup products in the
middle and bottom rows of (\ref{eq:bowow}) are compatible with the vertical homomorphisms
from the terms of the middle row to the terms of the bottom row.  The latter homomorphisms
are those associated to the natural morphism $\mathbb{Z} \to \mathbb{Z}/n$
of \'etale sheaves on $X$, since $H^i(X,\mu_n) = \mathrm{Ext}^i_X(\mathbb{Z},\mu_n)$
and the terms of the middle row have the form $\mathrm{Ext}^i_X(\mathbb{Z}/n,\mu_n)$.
So the above compatibility of the middle and bottom rows follows from the naturality of cup product (\ref{eq:cupwithpureinjective}) with respect 
to morphisms of the arguments. Note that in showing this, we have not used any compatibility of cup products 
with boundary maps; the latter requires more hypotheses.

We now turn to analyzing the connection of the cup product pairing (\ref{eq:natural})
$$H^1(X,\mu_n) \times H^1(X,\mu_n) \to H^2(X,\mu_n^{\otimes 2})$$
with the diagram (\ref{eq:bowow}).  We are to prove that this is compatible with 
pulling back and pushing forward arguments to the second row of (\ref{eq:bowow}).

By the naturality of cup product pairings with respect to either argument, we have a commuting diagram of pairings
\begin{equation}
\label{eq:bowow2}
\xymatrix @C=.2pc{
 \mathrm{Ext}^1_X(\mathbb{Z}/n,\mu_n)\ar[d] & \times&H^1(X,\mu_n)\ar@{=}[d]&\ar[rrr] &&&& \mathrm{Ext}^2_X(\mathbb{Z}/n,\mu_n^{\otimes 2})\ar[d]\\
 H^1(X,\mu_n) &\times &H^1(X,\mu_n)&\ar[rrr] &&&& H^2(X,\mu_n^{\otimes 2})
}
\end{equation}
in which the left and right vertical homomorphisms are induced by the canonical surjection $\mathbb{Z} \to \mathbb{Z}/n$.
We claim that the top row of this diagram fits into a diagram of pairings
\begin{equation}
\label{eq:bowow3}
\xymatrix @C=.2pc{
 \mathrm{Ext}^1_X(\mathbb{Z}/n,\mu_n)\ar@{=}[d] & \times&H^1(X,\mu_n)\ar[d]&\ar[rrr] &&&& \mathrm{Ext}^2_X(\mathbb{Z}/n,\mu_n^{\otimes 2})\ar[d]\\
 \mathrm{Ext}^1_X(\mathbb{Z}/n,\mu_n) & \times&\mathrm{Ext}^2_X(\mathbb{Z}/n,\mu_n)&\ar[rrr] &&&&
\mathrm{Ext}^3_X(\mathbb{Z}/n,\mu_n^{\otimes 2})
}
\end{equation}
that commutes up to the sign $(-1)$ and 
in which the  middle vertical map is the boundary map resulting from the sequence
\begin{equation}
\label{eq:Bseq}
B_\bullet = (0 \to B'' \to B \to B' \to 0)  = ( 0 \to \mathbb{Z} \xrightarrow{\cdot n} \mathbb{Z} \to \mathbb{Z}/n\to 0)
\end{equation}
and the right vertical map is the boundary map  associated with the Bockstein sequence
\begin{equation}
\label{eq:bock}
C_\bullet = (0 \to C'' \to C \to C' \to 0) = (0  \to \mathbb{Z}/n \to \mathbb{Z}/n^2 \to \mathbb{Z}/n \to 0).
\end{equation}
Let $A = \mathbb{Z}/n$.  We have a morphism $C_\bullet \to A \otimes B_\bullet$ fitting into a commutative diagram 
\begin{equation}
\label{eq:bowow4}
\xymatrix {
&\mathbb{Z}/n \ar[r]^{0}\ar@{=}[d] &\mathbb{Z}/n\ar[r]^{1} \ar@{=}[d]&\mathbb{Z}/n\ar[r]\ar@{=}[d]&0\;\\
&A \otimes B'' \ar[r]&A \otimes B\ar[r]&A \otimes B'\ar[r]&0\;\\
0\ar[r]&C'' \ar[r]\ar[u]\ar@{=}[d] &C\ar[r]\ar[u] \ar@{=}[d]&C'\ar[r]\ar[u]\ar@{=}[d]&0\;\\
0\ar[r]&\mathbb{Z}/n \ar[r] &\mathbb{Z}/n^2\ar[r] &\mathbb{Z}/n\ar[r]&0.
}
\end{equation}
Choosing pure injective resolutions $\mu_n\to I^\bullet$ and $\mu_n^{\otimes 2}\to K^\bullet$ and a morphism of resolutions 
$I^\bullet\otimes I^\bullet\to K^\bullet$ over $\mu_n^{\otimes 2}$, we can apply the respective Hom functors over $X$
to the diagram (\ref{eq:bowow4}) to obtain a commutative diagram 
\begin{footnotesize}
$$\xymatrix @C=1.2pc {
0&\mathrm{Hom}_X(A,I^\bullet)\otimes \mathrm{Hom}_X(B'',I^\bullet)\ar[l]\ar[d] &\mathrm{Hom}_X(A,I^\bullet)\otimes \mathrm{Hom}_X(B,I^\bullet)\ar[l]\ar[d]&
\mathrm{Hom}_X(A,I^\bullet)\otimes \mathrm{Hom}_X(B',I^\bullet)\ar[l]\ar[d]&\\
&\mathrm{Hom}_X(A\otimes B'',I^\bullet\otimes I^\bullet)\ar[d] & \mathrm{Hom}_X(A\otimes B,I^\bullet\otimes I^\bullet)\ar[d]\ar[l]&
\mathrm{Hom}_X(A\otimes B',I^\bullet\otimes I^\bullet)\ar[d]\ar[l]&0\;\ar[l]\\
0& \mathrm{Hom}_X(C'',K^\bullet) \ar[l]&\mathrm{Hom}_X(C,K^\bullet)\ar[l] & \mathrm{Hom}_X(C',K^\bullet)\ar[l]&\ar[l] 0.
}$$
\end{footnotesize}
It follows from \cite[Lemma 3.2]{Swan} that the diagram (\ref{eq:bowow3}) commutes up to the sign $(-1)$.

In view of diagrams (\ref{eq:bowow2}) and (\ref{eq:bowow4}), the last assertion (\ref{eq:cup}) of
 Theorem \ref{thm:bigdiagram}  concerning the relation of (\ref{eq:natural}) to
 the pairing in the middle row of (\ref{eq:bowow3}) will hold if we can show the following
 assertion.  We claim that the 
rightmost  vertical homomorphism
 $$\lambda: \mathrm{Ext}^2_X(\mathbb{Z}/n,\mu_n^{\otimes 2})\to \mathrm{Ext}^3_X(\mathbb{Z}/n,\mu_n^{\otimes 2})$$
in (\ref{eq:bowow3}), which is  induced by the boundary map of the Bockstein sequence 
$C_\bullet$ in (\ref{eq:bock}), is the composition of the pullback map 
 $$\tau: \mathrm{Ext}^2_X(\mathbb{Z}/n,\mu_n^{\otimes 2})\to H^2(X,\mu_n^{\otimes 2})$$
 associated to $\mathbb{Z} \to \mathbb{Z}/n$ with the boundary map 
 $$\nu: H^2(X,\mu_n^{\otimes 2})\to \mathrm{Ext}^3_X(\mathbb{Z}/n,\mu_n^{\otimes 2})$$
 associated to the sequence $B_\bullet$ in (\ref{eq:Bseq}).
 
 This assertion (and the more general fact, which holds in all degrees) can be proved
 by calculating $\lambda$ and $\nu \circ \tau$ using a pure injective resolution
 of the second argument, which in this case is $\mu_n^{\otimes 2}$.  To be explicit, let 
 $$0\to \mu_n^{\otimes 2}\to K^0\to K^1\to K^2\to K^3\to\cdots$$ 
 be a pure injective resolution. The boundary map $\lambda$
 results from taking elements of $\mathrm{Hom}_X(\mathbb{Z}/n,K^2)$ which go
 to zero in $K^3$, lifting these to elements of $\mathrm{Hom}_X(\mathbb{Z}/n^2,K^2)$
 by the injectivity of $K^2$, and then pushing this lift forward by $K^2 \to K^3$ to produce an element of $\mathrm{Hom}_X(\mathbb{Z}/n,K^3)$.  The map $\tau$ results from simply inflating a homomorphism
 in $\mathrm{Hom}_X(\mathbb{Z}/n,K^2)$ to one in $\mathrm{Hom}_X(\mathbb{Z},K^2)$ via the natural surjection $\mathbb{Z} \to \mathbb{Z}/n$.
 The map $\nu$ results from lifting maps
 from $\mathrm{Hom}_X(\mathbb{Z},K^2)$ to $\mathrm{Hom}_X(\mathbb{Z},K^2)$ through the multiplication by $n$ homomorphism $\mathbb{Z} \xrightarrow{\cdot n} \mathbb{Z}$ and then pushing the lift forward by $K^2 \to K^3$ to produce an element of $\mathrm{Hom}_X(\mathbb{Z}/n, K^3)$.  Since we can use 
 the lifts involved in calculating $\lambda$ to do the calculations to find $\nu$ on maps which come from the inflation map $\tau$, we see that $\lambda = \nu\circ \tau$.

\section{A reformulation of the approach via Artin maps} 
\label{s:reformulate}
\setcounter{equation}{0}

We describe in this section our approach to proving Theorem \ref{thm:fixed?}.  Instead of the diagram of pairings (\ref{eq:cuppairnew}), we consider the
diagram of pairings 
\begin{equation}
\label{eq:cuppair}
\xymatrix @C=.2pc{
H^1(G,\tilde{\mu}_n)\ar[d]^{f^*_X} &\times &H^2(G,\tilde{\mu}_n)\ar[d]^{f^*_X} &\ar[rrr]&&&& H^3(G,\tilde{\mu}_n^{\otimes 2})&  = &H^3(G,\tilde{\mu}_n) \otimes \tilde{\mu}_n\ar[d]^{f^*_X}\\
 H^1(X,\mu_n) &\times &H^2(X,\mu_n) &\ar[rrr]&&&& H^3(X,\mu_n^{\otimes 2})& = & H^3(X,\mu_n) \otimes \tilde{\mu}_n = \tilde{\mu}_n}
\end{equation}
in which the vertical homomorphisms are induced by $f$.  Let   $\phi:\mathbb{Z}/n \to \tilde{\mu}_n$ be the isomorphism taking $1$ mod $n$ to $\zeta_n$, where $\zeta_n\in F$ is as in Definition \ref{def:roots}.
Then $\phi$ takes the generator $c_1$ of $H^1(G,\mathbb{Z}/n)$ to a generator $d_1 = \phi(c_1)$ of $H^1(G,\tilde{\mu}_n)$.We have
\begin{equation}
\label{eq:goal}
\phi(f^*_X(c)) = \phi(f^*_X(c_1) \cup f^*_X(c_2)) = f^*_X(\phi(c_1)) \cup f^*_X(c_2) = f^*_X(d_1) \cup f_X^*(c_2).
\end{equation}
We will show (\ref{eq:calculation}) of Theorem \ref{thm:fixed?} by calculating
the cup product of $f_X^*(d_1)$ and $f_X^*(c_2)$  using Mazur's description in \cite{Mazur} of
the bottom row of (\ref{eq:cuppair}).  

\section{Analysis of $f_X^*(d_1)$}
\label{s:analysis}
\setcounter{equation}{0}

\begin{lemma}
\label{lem:first} There is a canonical isomorphism $$H^1(X,\mu_n) = \mathrm{Hom}(\mathrm{Pic}(X),\tilde{\mu}_n).$$  The restriction of a class $d \in H^1(X,\mu_n)$ to $H^1(\mathrm{Spec}(F),\mu_n)$ defines a torsor $Y(d)$ for the group scheme $\schememu$ over $\mathrm{Spec}(F)$.  
The scheme $Y(d)$ is isomorphic to $\mathrm{Spec}(\frac{F[w]}{(w^n - \xi)})$
as a $\schememu$ torsor for an element $\xi \in F^*$ which is unique up to multiplication by an element of $(F^*)^n$.
\end{lemma}

\begin{proof} Our choice of a primitive $n^{th}$ root of unity $\zeta_n$ in $F$ gives an isomorphism of \'etale sheaves from $\mathbb{Z}/n$ to $\mu_n$.  This induces an isomorphism from $H^1(X,\mu_n)$ to $H^1(X,\mathbb{Z}/n)$.  The group
$H^1(X,\mathbb{Z}/n)$ classifies torsors for the constant group scheme $\mathbb{Z}/n$.  Therefore
$$H^1(X,\mathbb{Z}/n) = \mathrm{Hom}(\pi_1(X),\mathbb{Z}/n) = \mathrm{Hom}(\mathrm{Pic}(X),\mathbb{Z}/n)$$ where the last isomorphism results from class field theory.  Thus 
\begin{eqnarray*}
H^1(X,\mu_n) &=& H^1(X,\mathbb{Z}/n) \otimes_{\mathbb{Z}} \tilde{\mu}_n \\
&=& \mathrm{Hom}(\mathrm{Pic}(X),\mathbb{Z}/n)  \otimes_{\mathbb{Z}} \tilde{\mu}_n \\
&=& \mathrm{Hom}(\mathrm{Pic}(X),\tilde{\mu}_n)
\end{eqnarray*}
and  the isomorphism between the far left and far right terms does not depend on the choice of $\zeta_n$. The last statement is clear from Kummer theory over fields of characteristic $0$;  see \cite[p. 125, Thm. 3.9]{Milne}.
\end{proof}

 \begin{remark}
 \label{rem:Kummer2}
 Suppose the class $d\in H^1(X,\mu_n)$ has order $n$. Then $Y(d) = \mathrm{Spec}(K)$ for an everywhere unramified $\mathbb{Z}/n$ extension $K = F(\xi^{1/n})$ of $F$ 
 for an element $\xi \in F^*$ as in Lemma \ref{lem:first}.  Associating $d$ canonically to 
 a homomorphism $d:\mathrm{Pic}(X) \to \tilde{\mu}_n$ as in Lemma \ref{lem:first}, the element $\xi$ has the property
 that 
 \begin{equation}
 \label{eq:normalize}
 \frac{\mathrm{Art}(a)(\xi^{1/n})}{\xi^{1/n} } = d(a)\quad \mathrm{for \ all}\quad a \in \mathrm{Pic}(X)
 \end{equation}
where
 $\mathrm{Art}(a) \in \mathrm{Gal}(K/F)$ is the image of $a \in \mathrm{Pic}(X) $ under the Artin map.  The equality (\ref{eq:normalize}) does not depend on the choice of $n^{th}$ root $\xi^{1/n}$ of $\xi$ in $K$.  It specifies the class of $\xi$ uniquely
 in the quotient group $F^*/(F^*)^n$.
 \end{remark}

\begin{lemma}
\label{lem:H2picture}
The Pontryagin dual $H^2(X,\mathbb{Z}/n)^\star = \mathrm{Hom}(H^2(X,\mathbb{Z}/n), \mathbb{Q}/\mathbb{Z})$ of $H^2(X,\mathbb{Z}/n)$ lies in an exact sequence
\begin{equation} 
\label{eq:bound}
1 \to O_F^*/(O_F^*)^n  \xrightarrow{\;\tau\;} H^2(X,\mathbb{Z}/n)^\star \xrightarrow{\;\delta\;} \mathrm{Pic}(X)[n] \to 0
\end{equation}
in which $\mathrm{Pic}(X)[n]$ is the $n$-torsion in $\mathrm{Pic}(X)$.  Define
$T$ to be the subgroup of $\gamma \in F^*$ such that $\gamma O_F$ is the $n^{th}$
power of some fractional ideal $I(\gamma)$.  Then there is a canonical isomorphism
\begin{equation}
\label{eq:precise}
T/(F^*)^n =  H^2(X,\mathbb{Z}/n)^\star 
\end{equation}
with the following properties. 
\begin{enumerate}
\item[i.] The homomorphisms $\tau$ and $\delta$ in
(\ref{eq:bound}) are induced by the inclusion $O_F^* \subset T$ and the map which sends
$\gamma \in T$ to the ideal class $[I(\gamma)]$ of $I(\gamma)$.  
\item[ii.] The 
homomorphism $h:H^1(X,\mu_n) \to H^2(X,\mathbb{Z}/n)^\star$ induced by the cup product pairing
$$H^1(X,\mu_n) \times H^2(X,\mathbb{Z}/n) \to  H^3(X,\mu_n) =  \mathbb{Z}/ n\mathbb{Z}$$
has the following description.  Suppose $d \in H^1(X,\mu_n)$
gives a $\schememu$ torsor $Y(d)$ over $\mathrm{Spec}(F)$ as in Lemma \ref{lem:first}. Let $\xi \in F^*$ be associated to 
$Y(d)$ as in Lemma \ref{lem:first}, so that $\xi$ is unique up to multiplication
by an element of $(F^*)^n$.  Then $\xi \in T$, and $h(d)$ is the coset $\xi (F^*)^n$ in
$T/(F^*)^n = H^2(X,\mathbb{Z}/n)^\star$.
\end{enumerate}
\end{lemma}

\begin{proof}  The exact sequence (\ref{eq:bound}) is shown in \cite[p. 539]{Mazur}.  This utilizes 
Artin-Verdier duality (c.f. \cite[p. 538]{Mazur}), which gives a canonical isomorphism
$H^2(X,\mathbb{Z}/n)^\star = \mathrm{Ext}^1_X(\mathbb{Z}/n,G_{m,X})$.  The more precise
description in (\ref{eq:precise}), together with properties in (i) and (ii) of this description, results from
the analysis of $\mathrm{Ext}^1_X(\mathbb{Z}/n,G_{m,X})$ and the computation of 
duality pairings by Hilbert symbols in \cite[p. 540-541]{Mazur}.
\end{proof}

\begin{corollary}
\label{cor:c1analysis}  Suppose $d_1$ is a generator of $H^1(G,\tilde{\mu}_n)$. The class $d = f_X^*(d_1) \in H^1(X,\mu_n)$ corresponds to a $\schememu$-torsor $Y(d) = \mathrm{Spec}(K)$ over $\mathrm{Spec}(F)$ such that $K = F(\xi^{1/n})$ of $F$ for an element $\xi \in F^*$ with the following properties.
\begin{enumerate}
\item[i.] The extension $K/F$ is everywhere unramified and cyclic of degree $n$.  Fixing an embedding of $K$ into the maximal unramified extension $F^{un}$ of $F$ determines a surjection $\rho: \mathrm{Gal}(F^{un}/F) = \pi_1(X,\eta) \to \mathrm{Gal}(K/F)$.  
\item[ii.] There is a unique isomorphism $\lambda:\mathrm{Gal}(K/F)
\to G = \mathbb{Z}/n$ such that $\lambda \circ \rho:\pi_1(X,\eta) \to G$ is the homomorphism
$f:\pi_1(X,\eta) \to G$ used to construct Kim's invariant.  
\item[iii.] The element $\xi \in F^*$ is uniquely determined mod $(F^*)^n$ by the requirement that  (\ref{eq:normalize}) hold when we identify $d$ with an element of
$\mathrm{Hom}(\mathrm{Pic}(X),\tilde{\mu}_n)$ as in Lemma \ref{lem:first}.
\item[iv.] The image of $d = f_X^*(d_1)$ under the homomorphism
$h:H^1(G,\tilde{\mu}_n) \to H^2(G,\mathbb{Z}/n)^\star = T/(F^*)^n$ of Lemma \ref{lem:H2picture}
is the coset $\xi (F^*)^n$.
\end{enumerate}
\end{corollary}

\section{Hilbert pairings, Artin maps and $H^2(X,\mu_n)$}
\label{s:H2computation}
\setcounter{equation}{0}

With the notations of \S \ref{s:reformulate}, our goal is to compute the cup product
\begin{equation}
\label{eq:alright}
 f_X^*(d_1) \cup f_X^*(c_2) = h(f_X^*(d_1)) ( f_X^*(c_2)) \in \tilde{\mu}_n
\end{equation}
when $c_2$ is a generator of $H^2(G,\tilde{\mu}_n)$, $f_X^* c_2$ is the pullback
of $c_2$ to $H^2(X,\mu_n)$ and $h(f_X^*(d_1))$ is the element of the Pontryagin dual $H^2(X,\mathbb{Z}/n)^\star$
determined in Corollary \ref{cor:c1analysis}.  To do this, we first develop in this section a description of
$H^2(X,\mu_n) = H^2(X,\mathbb{Z}/n) \otimes \tilde{\mu}_n$ using ideles of $F$.

Let $j:\mathrm{Spec}(F) \to X$ be the inclusion of the generic point of $X$ into $X$. Then
$j_* \mu_{n,F} = \mu_{n,X}$ since $F$ contains a primitive $n^{th}$ root of unity.  There is a spectral
sequence
\begin{equation}
\label{eq:spectral}
H^p(X,R^q j_*(\mu_{n,X})) \to H^{p+q}(F,\mu_{n,F}).
\end{equation}
Consider the $(p,q) = (2,0)$ term.  This is associated to the restriction homomorphism
\begin{equation}
\label{eq:2zero}
H^2(X,R^0 j_*(\mu_{n,X})) = H^2(X,\mu_{n,X}) \to H^2(F,\mu_{n,F}).
\end{equation}
By the Kummer sequence 
\begin{equation}
\label{eq:KummerF}
0 \to \mu_{n,F} \to G_{m,F} \to G_{m,F} \to 0
\end{equation}
and Hilbert Theorem 90, the homomorphism $H^2(F,\mu_{n,F}) \to H^2(F,G_{m,F})$
is injective.  The composition of $H^2(X,\mu_{n,X}) \to H^2(F,\mu_{n,F})$ with this homomorphism
factors through the homomorphism $H^2(X,\mu_{n,X}) \to H^2(X,G_{m,X})$.  However,
elements of $H^2(X,G_{m,X})$ are elements of the Brauer group of $F$ with trivial local invariants everywhere since $F$ is totally complex, and such elements must be trivial.  Thus $H^2(X,G_{m,X}) = \{0\}$ and it follows that
(\ref{eq:2zero}) is the zero homorphism.  Hence in the spectral sequence (\ref{eq:spectral})
gives an exact sequence
\begin{equation}
\label{eq:nicejob}
H^1(F,\mu_{n,F}) \to H^0(X,R^1 j_* \mu_n) \xrightarrow{\;\omega\;} H^2(X,\mu_{n,X}) \to 0.
\end{equation}
The homomorphism $\omega$  can be realized in the following way (up to a possibly
multiplying by $-1$, depending on one's conventions for boundary maps in spectral sequences). 
Taking the long exact sequence associated to the functor $j_*$ applied to (\ref{eq:KummerF}) 
gives an exact sequence
 \begin{equation}
 \label{eq:longj}
 0 \to \mu_{n,X} \to j_* G_{m,F} \to j_* G_{m,F} \to R^1 j_* \mu_{n,F} \to 0
 \end{equation}
 since $R^1 j_* G_{m,F} = 0$ by Hilbert Theorem 90. Splitting (\ref{eq:longj}) into
 two short exact sequences and then taking boundary maps in the associated long exact cohomology sequences over $X$ produces the transgression map $\omega$ in (\ref{eq:nicejob}) up to possibly multiplying by $-1$.
 
We now recall from \cite[p. 36-39]{Milne} some definitions.
 
 \begin{definition}
 \label{lem:notations}  Let $x$ be a point of $X$ with residue field $k(x)$. Define $O_x = O_{X,x}$ to be
 the local ring of $x$ on $X$.  Let $\overline{x}$ be a geometric point of $X$ over $x$, so that $k(\overline{x})$ is a separable
 closure of $k(x)$.  The Henselization $O_{x,h}$ of $O_x$ (resp. the strict Henselization $O_{x,sh}$ of $O_x$) is the direct limit  of all 
 of all local rings $D$ (resp. $D'$) which are \'etale $O_x$-algebras having residue field $k(x)$ (resp. having
 residue field inside $\overline{k(x)}$).  Let $\hat{O}_x$ be the completion of $O_x$ and let
 $\hat{O}_{\overline{x}}$ be the direct limit of all finite \'etale local $\hat{O}_x$ algebras
 having residue field in $\overline{k(x)}$.  
 \end{definition}

  The following result is implicit in \cite{Mazur}, but we will recall the argument since the details of the computation enter into some later calculations. 
    
 \begin{lemma}
 \label{lem:r1comp}  Let $x$ be a point of $X$, and let $\overline{x}$ be a geometric point over $x$.
   The stalk $(R^1 j_* \mu_n) |_{\overline{x}}$ of $R^1 j_* \mu_n$ at $\overline{x}$ is the cohomology group $H^1(F_{x,sh},\mu_n)$, where $F_{x,sh} = F\otimes_{O_F} O_{x,sh}$.  The Kummer sequence
  $$1 \to \mu_n \to G_m \to G_m \to 1$$
  over $F_{x,sh}$ is exact.  The $\mathrm{Gal}(\overline{F}_{x,sh}/F_{x,sh})$ cohomology of this sequence gives
  an isomorphisms
 \begin{equation}
 \label{eq:fx}
 F_{x,sh}^* / (F_{x,sh}^*)^n = H^1(F_{x,sh},\mu_n) = (R^1 j_* \mu_n) |_{\overline{x}}.
 \end{equation}
 This group is trivial if $x$ is the generic point of $X$. Suppose now that $x$ is a closed point, with residue field $k(x)$.  We then have natural isomorphisms 
 $\mathrm{Gal}(F_{x,sh}/F_{x,h}) = \mathrm{Gal}(\overline{k(x)}/k(x)) = \hat {\mathbb{Z}}$ where 
 $F_{x,h} = F \otimes_{O_F} O_{x,sh}$. One has
 \begin{equation}
 \label{eq:local}
 H^0(\mathrm{Gal}(F_{x,sh}/F_{x,h}),  (R^1 j_* \mu_n) |_{\overline{x}} ) = \hat{F}_x^*/T_x
 \end{equation}
 where $\hat{F}_x = \mathrm{Frac}(\hat{O}_x)$ is the completion of $F$ with respect to that discrete absolute value at $x$ and 
 $T_x \supset (\hat{F}_x^*)^n$ is the subgroup of $\gamma \in \hat{F}_x^*$ such that $\hat{F}_x(\gamma^{1/n})$ is unramified over $\hat{F}_x$. Here $T_x/((\hat{F}_x)^*)^n$ is cyclic of order $n$.  Finally,
 \begin{equation}
 \label{eq:global}
 H^0(X,R^1 j_* \mu_{n,F}) = \bigoplus_{x \in X^0} H^0(\mathrm{Gal}(F_{x,sh}/F_{x,h}),  (R^1 j_* \mu_n) |_{\overline{x}} )  = \bigoplus_{x \in X^0} \hat{F}_x^*/T_x
 \end{equation}
 where $X^0$ is the set of closed points of $X$.
 \end{lemma}
 
\begin{proof}  The isomorphism (\ref{eq:fx}) results from the description of stalks of higher direct images in  \cite[Thm 1.15]{Milne} together with the long exact cohomology sequence of the Kummer sequence over
$F_{x,sh}$.  If $x$ is the generic point of $X$, then $F_{x,sh}$ is an algebraic closure of $F$ and the groups in (\ref{eq:fx}) are trivial. Suppose now that $x$ is a closed point.  We then have two exact sequences
\begin{equation}
\label{eq:two1}
1 \to \tilde{\mu}_n \to F_{x,sh}^* \to (F_{x,sh}^*)^n \to 1
\end{equation}
and
\begin{equation}
\label{eq:two2}
1 \to (F_{x,sh}^*)^n \to F_{x,sh}^* \to F_{x,sh}^*/ (F_{x,sh}^*)^n \to 1.
\end{equation}
Taking the cohomology of the second exact sequence (\ref{eq:two2}) with respect to $\Gamma = \mathrm{Gal}(F_{x,sh}/F_{x,h}) = \mathrm{Gal}(\overline{k(x)}/k(x))$ and then taking completions gives an exact sequence 
\begin{eqnarray}
\nonumber
0 \to ((\hat{F}^{un}_x)^*)^n)^{\Gamma} \to \hat{F}^*_x \to (F_{x,sh}^*/ (F_{x,sh}^*)^n)^\Gamma &&\\
\label{eq:yes} \to H^1(\Gamma,(F_{x,sh}^*)^n) \to H^1(\Gamma,F_{x,sh}^*) = 0 &&
\end{eqnarray}
where $\hat{F}_x^{un}$ is the maximal unramified extension of the complete local field $\hat{F}_x$.
The $\Gamma$-cohomology of the first exact sequence (\ref{eq:two1}) gives  
\begin{equation}
\label{eq:yestwo}
0 = H^1(\Gamma,F_{x,sh}^*) \to H^1(\Gamma,(F_{x,sh}^*)^n) \to H^2(\Gamma,\tilde{\mu}_n).
\end{equation}
The cohomology of finite modules for  $\Gamma = \hat{\mathbb{Z}}$ is trivial above dimension $1$.
So (\ref{eq:yestwo}) shows $H^1(\Gamma,(F_{x,sh}^*)^n) = 0$. In (\ref{eq:yes}), the group
$((\hat{F}^{un}_x)^*)^n)^{\Gamma}$ consists of those $\gamma \in \hat{F}_x^*$ such that $\hat{F}_x(\gamma^{1/n})$
is unramified over $\hat{F}_x$, so $((\hat{F}^{un}_x)^*)^n)^{\Gamma} = T_x$.  Hence (\ref{eq:yes}) now
shows (\ref{eq:local}).  

Now $R^1 j_* \mu_n$ has trivial stalk over the generic point of $X$, and units are $n^{th}$ powers locally in the \'etale topology over all closed points $x \in X^0$ having residue fields prime to $n$. We conclude from (\ref{eq:longj}) that $R^1 j_* \mu_{n,F}$ is the sheaf resulting from the direct sum of the stalks $(R^1 j_* \mu_n) |_{\overline{x}}$ as $x$ ranges over  $X^0$, from which   (\ref{eq:global}) follows.
\end{proof}

\begin{corollary}
\label{cor:longer}
The exact sequence (\ref{eq:nicejob}) is identified with
\begin{equation}
\label{eq:evenbetter}
F^*/(F^*)^n \xrightarrow{\;r\;} \bigoplus_{x \in X^0} \hat{F}_x^*/T_x \xrightarrow{\;\omega\;} H^2(X,\mu_n) \to 0.
\end{equation}
\end{corollary}

\begin{proof}  By the Kummer sequence over $F$ we have $H^1(F,\mu_n) = F^*/(F^*)^n$.
If $\beta \in F^*$, then $F(\beta^{1/n})$ is unramified at almost all places of $F$, so
$\beta \in T_x$ for all but finitely many $x \in X^0$.  Thus the natural homomorphisms
$F^* \to \hat{F}_x^*/T_x$ give rise to a homomorphism $r$ as in (\ref{eq:evenbetter}), and the
constructions in Lemma \ref{lem:r1comp}  identify $r$ with the first map in  (\ref{eq:nicejob}).
\end{proof}

\begin{lemma}
\label{lem:cupcalc} Suppose that in the description $H^2(X,\mathbb{Z}/n)^\star = T/(F^*)^n$
of Lemma \ref{lem:H2picture} we are given an element $\eta \in T$ describing a class
$\eta (F^*)^n \in T/(F^*)^n$.  Let $j \in J(F)$ be an idele of $F$ such that the 
component $j_x$ of $j$ at almost all $x \in X^0$ lies in $T_x$, so that $j$
defines an element $z(j)$ of $\oplus_{x \in X^0} (\hat{F}_x^*/T_x)$.  Then Corollary
\ref{cor:longer} produces an element $\omega(z(j))$ of $H^2(X,\mu_n)$.  We have
$H^2(X,\mu_n) = H^2(X,\mathbb{Z}/n) \otimes_{\mathbb{Z}} \tilde{\mu}_n$ and thus a natural
non-degenerate pairing
\begin{equation}
\label{eq:h2now}
\langle \ , \ \rangle: H^2(X,\mathbb{Z}/n)^\star \times H^2(X,\mu_n) \to \tilde{\mu}_n
\end{equation}
resulting from Pontryagin duality pairing
$$H^2(X,\mathbb{Z}/n)^\star \times H^2(X,\mathbb{Z}/n) \to \mathbb{Z}/n.$$
The value of the pairing in (\ref{eq:h2now}) on the pair $\eta (F^*)^n$ and $\omega(z(j))$ is
\begin{equation}
\label{eq:angle1}
\langle \eta (F^*)^n,\omega(z(j)) \rangle = Art(j)(\eta^{1/n})/\eta^{1/n}
\end{equation}
where $Art(j)$ is the image of $j$ under the Artin map $J(F) \to \mathrm{Gal}(F^{ab}/F)$
when $F^{ab} \supset F(\eta^{1/n})$ is the maximal abelian extension of $F$.
\end{lemma}

\begin{proof}  This follows from reducing the computation of duality pairings to the computation
of Hilbert symbols, as in \cite[\S 2.4-2.6]{Mazur}.  Here is one way to carry this out explicitly.

We have a 
long exact relative cohomology sequence 
\begin{eqnarray}
\nonumber
H^1(X,\mathbb{Z}/n) \to H^1(X - V, \mathbb{Z}/n) \to H^2_V(X,\mathbb{Z}/n)\to H^2(X,\mathbb{Z}/n) &&\\ 
\label{eq:excision}
 \xrightarrow{\;e\;} H^2(X-V,\mathbb{Z}/n) \to H^3_V(X,\mathbb{Z}/n) \xrightarrow{\;b\;}
H^3(X,\mathbb{Z}/n) &&\qquad
\end{eqnarray}
associated to a choice of a finite non-empty set $V$ of closed points of $X$ which is discussed in 
\cite[\S 2.5]{Mazur}.  

Suppose we take $V$ large enough so that $\mathrm{Pic}(X-V) = 0$ and all of the residue
characteristics of points of $X-V$ are relatively prime to $n$.  Then the Kummer sequence
$$1 \to \mu_{n,X-V} \to G_{m,X-V} \to  G_{m,X-V} \to 1$$
is exact.  So $$H^1(X-V,G_{m,X-V}) = \mathrm{Pic}(X-V) = 0$$ implies $H^2(X-V,\mu_{n,X-V})$ equals the $n$-torsion in the Brauer group $H^2(X-V,G_{m,X-V})$.  This $n$-torsion has order
$n^{\#V - 1}$ by the usual theory of elements of the Brauer group of $F$ which are unramified
outside of $V$.  By local duality (c.f. \cite[p. 540, 538]{Mazur}), 
$$H^3_V(X,\mathbb{Z}/n) =  \prod_{P\in V} \tilde{\mu}_n^\star.$$
Global duality gives 
$$H^3(X,\mathbb{Z}/n) = \mathrm{Ext}_X^0(\mathbb{Z}/n,G_m) = \tilde{\mu}_n^\star.$$
By considering the orders of these groups, we see that the map $b$ in (\ref{eq:excision}) has
kernel exactly $H^2(X-V,\mathbb{Z}/n)$, so the map $e$ is trivial.  

By local duality (\textit{op. cit.}) we have
$$H^2_V(X,\mathbb{Z}/n) = \prod_{P \in V} (\hat{O}_P^*/(\hat{O}_P^*)^n)^\star.$$ 
Using these isomorphisms in (\ref{eq:excision})  and taking Pontryagin duals gives an exact sequence
\begin{equation}
\label{eq:describe}
0 \to H^2(X,\mathbb{Z}/n)^\star \to \prod_{P \in V} \hat{O}_P^*/(\hat{O}_P^*)^n \to H^1(X - V, \mathbb{Z}/n)^\star \to H^1(X,\mathbb{Z}/n)^\star.
\end{equation}
By class field theory, $$H^1(X,\mathbb{Z}/n) = \mathrm{Hom}(Cl(O_F),\mathbb{Z}/n)$$ 
and $$H^1(X- V,\mathbb{Z}/n) = \mathrm{Hom}(Cl_{m_V}(O_F),\mathbb{Z}/n)$$ when
$Cl(O_F) = \mathrm{Pic}(O_F)$ is the ideal class group of $O_F$ and $Cl_{m_V}(O_F)$
is the ray class group of conductor $m_V$ for $m_V$ a sufficiently high power of the product
of the prime ideals of $O_F$ corresponding to  $P \in V$.  

Thus (\ref{eq:describe}) becomes
\begin{equation}
\label{eq:describetwo}
0 \to H^2(X,\mathbb{Z}/n)^\star \to \prod_{P \in V} \hat{O}_P^*/(\hat{O}_P^*)^n \to \frac{Cl_{m_V}(O_F)}{nCl_{m_V}(O_F)} \to \frac{Cl(O_F)}{nCl(O_F)}
\end{equation}
where the right hand homomorphism is induced by the canonical surjection $Cl_{m_V}(O_F) \to Cl(O_F)$.

Now in (\ref{eq:evenbetter}), since $H^2(X,\mu_n)$ is finite, we can take $V$ as above sufficiently large so that there is a surjection
\begin{equation}
\label{eq:allright}
\bigoplus_{P \in V \subset X^0} \hat{F}_P^*/T_P \to H^2(X,\mu_n) \to 0.
\end{equation}
The compatibility of local and global duality pairings shows that  pairing
$$H^2(X,\mathbb{Z}/n)^\star \times H^2(X,\mu_n) \to \tilde{\mu}_n$$
in (\ref{eq:h2now}) 
results from (\ref{eq:describetwo}), (\ref{eq:allright}) and the pairings
\begin{equation}
\label{eq:dual}
\frac{\hat{O}_P^*}{(\hat{O}_P^*)^n} \times \frac{\hat{F}_P^*}{T_P} \to \tilde{\mu}_n
\end{equation}
induced by the Hilbert pairings
\begin{equation}
\label{eq:dualtwo}
\frac{\hat{F}_P^*}{(\hat{F}_P^*)^n} \times \frac{\hat{F}_P^*}{(\hat{F}_P^*)^n} \to \tilde{\mu}_n.
\end{equation}
Note  here that (\ref{eq:dual}) is non-degenerate since (\ref{eq:dualtwo}) is non-degenerate and
$T_P/(\hat{F}_P^*)^n$ corresponds by class field theory to the unique cyclic unramified extension of 
degree $n$ of $\hat{F}_P$. 

This description of (\ref{eq:h2now}) leads to (\ref{eq:angle1}) by the compatibility of the Artin map with Hilbert pairings.  
\end{proof}

\section{Analysis of $f_X^* c_2$.}
\label{s:anotheranalysis}
\setcounter{equation}{0}

Our goal now is to compute the cup product in (\ref{eq:alright}) using Lemma \ref{lem:cupcalc}. We have a
reasonable description of $ h(f_X^*(d_1)) \in H^2(X,\mathbb{Z}/n)^\star$ from Corollary \ref{cor:c1analysis}
in terms of a Kummer generator $\xi \in F$ for  the $\schememu$-torsor $Y(f_X^*(d_1))$ produced by the generator $d_1 \in H^1(G,\tilde{\mu}_n)$ and  the homomorphism
$f:\pi_1(X) \to G$.  Recall that we assumed $f$ surjective, and we know $K = F(\xi^{1/n})$ is a cyclic degree $n$ Kummer extension which is everywhere unramified over $F$.  In this section we must develop an expression for $f_X^* c_2 \in H^2(X,\mu_n)$
when $c_2$ is a generator for $H^2(G,\tilde{\mu}_n)$. This will then be used in Lemma \ref{lem:cupcalc}.

Consider the exact sequences of $G = \mathrm{Gal}(K/F)$-modules
\begin{equation}
\label{eq:Kseq1}
1 \to \tilde{\mu}_n \to K^* \to (K^*)^n \to 1
\end{equation}
and
\begin{equation}
\label{eq:Kseq2}
1 \to (K^*)^n \to K^* \to K^*/(K^*)^n \to 1.
\end{equation}

\begin{lemma}
\label{lem:boundary}  The composition of the boundary maps in the  long exact $G$-cohomology sequences associated to (\ref{eq:Kseq1}) and (\ref{eq:Kseq2}) gives an exact sequence
\begin{equation}
\label{eq:onto}
F^* \to (K^*/(K^*)^n)^G \to H^2(G,\tilde{\mu}_n) \to 0.
\end{equation}
Here $H^2(G,\tilde{\mu}_n)$ is cyclic of order $n$. So there is a $\gamma \in K^*$ such that
the coset $\gamma (K^*)^n$ is in $(K^*/(K^*)^n)^G$, and the image of this coset in $H^2(G,\tilde{\mu}_n)$
equals the generator $c_2$.  
\end{lemma} 

\begin{proof}
Since $G$ is cyclic, the map $H^2(G,\tilde{\mu}_n) \to H^2(G,K^*)$ is the cup product with a generator of $H^2(G,\mathbb{Z})$ of the map $\hat{H}^0(G,\tilde{\mu}_n) \to \hat{H}^0(G,K^*)$ of Tate cohomology groups.  
Since $K/F$ is everywhere unramified, every element of $\tilde{\mu}_n$ is a local norm.  Therefore every element of $\tilde{\mu}_n$ is a global norm from $K$ to $F$ because $K/F$ is cyclic.  Therefore $\hat{H}^0(G,\tilde{\mu}_n) \to \hat{H}^0(G,K^*)$ is the trivial map, so $H^2(G,\tilde{\mu}_n) \to H^2(G,K^*)$ is the trivial map. Because $H^1(G,K^*) = 0$,  the $G$ cohomology of the exact sequences (\ref{eq:Kseq1}) and (\ref{eq:Kseq2})
gives (\ref{eq:onto}).
\end{proof}

\begin{lemma}
\label{lem:nicegamma}  With the above notations, the extension $K(\gamma^{1/n})$ is a cyclic
degree $n^2$ extension of $F$ which contains $K$.  There is an idele $j = (j_v)_v$ of $J(F)$
with the following properties.  If $v$ is an infinite place of $F$, $j_v = 1$.  Suppose
$v$ is finite and that $v$ corresponds to the closed point $x$ of $X$.  Then for all places $w$ of
$K$ above $v$, the images of $j_v \in F_v^* = \hat{F}_x^*$ and $\gamma \in K^*$ in 
$$(\hat{F}_{x,sh}^*/(\hat{F}_{x,sh}^*)^n)^{G_x} =  \hat{F}_x^*/T_x$$
agree for any embedding of $K_w$ into $\hat{F}_{x,sh}$ over $F_v$.  Let $\mathrm{ord}_v:F_v^* \to \mathbb{Z}$ be the discrete valuation associated to $v$.  Then $\mathrm{ord}_v(\mathrm{Norm}_{K/F}(\gamma))$ lies in $n \mathbb{Z}$, and there is a congruence of integers 
\begin{equation}
\label{eq:wcongruence}
\frac{\mathrm{ord}_v(\mathrm{Norm}_{K/F}(\gamma))}{n} \equiv \mathrm{ord}_v(j_v) \equiv \mathrm{ord}_w(\gamma) \quad \mathrm{mod} \quad  n_v \mathbb{Z}
\end{equation}
when $n_v$ is the order the decomposition group $G_{w} \subset G = \mathrm{Gal}(K/F)$ of any place $w$ over $v$ in $K$ and $\mathrm{ord}_w:K_w^* \to \mathbb{Z}$ is the discrete valuation associated to $w$.  Finally, in the notation of Lemma \ref{lem:cupcalc}, for all such $j$ 
the element $\omega(z(j)) \in H^2(X,\mu_n)$ equals $f_X^* c_2$.
\end{lemma}

\begin{proof} Since $\gamma (K^*)^n$ lies in the invariants $(K^*/(K^*)^n)^G$, and $G = \mathrm{Gal}(K/F)$ is cyclic, the extension
$K(\gamma^{1/n})$ is abelian over $F$.  Since $\gamma (K^*)^n$ has image of order $n$ in
$H^2(G,\tilde{\mu}_n)$, it must define an element of $K^*/(K^*)^n$ of order $n$.  So $K(\gamma^{1/n})$
is a cyclic degree $n$ extension of $K$.

By Kummer theory, 
$$K^*/(K^*)^n = \mathrm{Hom}_{cont}(\mathrm{Gal}(K^{(n)}/K),\tilde{\mu}_n)$$
and
$$F^*/(F^*)^n = \mathrm{Hom}_{cont}(\mathrm{Gal}(F^{(n)}/F),\tilde{\mu}_n)$$
when $K^{(n)}$ is the maximal abelian exponent $n$ extension of $K$ and $F^{(n)}$ is defined similarly for
$F$.  The natural homomorphism $F^*/(F^*)^n \to K^*/(K^*)^n$ corresponds to  the  map
$$\mathrm{Hom}_{cont}(\mathrm{Gal}(F^{(n)}/F),\tilde{\mu}_n) \to \mathrm{Hom}_{cont}(\mathrm{Gal}(K^{(n)}/K),\tilde{\mu}_n)$$
which results from restricting homomorphisms from $\mathrm{Gal}(F^{(n)}/F)$ to $\mathrm{Gal}(F^{(n)}/K)$ and then inflating them to $\mathrm{Gal}(K^{(n)}/K)$.  The image of $$F^*/(F^*)^n \to K^*/(K^*)^n$$ is thus contained in the set $\mathcal{H}$ of  those elements of $\mathrm{Hom}_{cont}(\mathrm{Gal}(K^{(n)}/K),\tilde{\mu}_n)$ which are inflated from elements of the group $\mathrm{Hom}_{cont}(\mathrm{Gal}(F^{(n)}/K),\tilde{\mu}_n)$. Let us show that this image is precisely $\mathcal{H}$. It is enough to show that any continuous homomorphism
$\mathrm{Gal}(F^{(n)}/K)\to \tilde{\mu}_n$ can be extended to a continuous homomorphism
$\mathrm{Gal}(F^{(n)}/F)\to \tilde{\mu}_n$. This is so  because the sequence
$$1 \to \mathrm{Gal}(F^{(n)}/K) \to \mathrm{Gal}(F^{(n)}/F) \to \mathrm{Gal}(K/F) \to 1$$
splits owing to the fact that $\mathrm{Gal}(F^{(n)}/F)$ is an exponent $n$ abelian group and 
$\mathrm{Gal}(K/F)$ is isomorphic to $\mathbb{Z}/n$.  

In view of (\ref{eq:onto}) and the above discussion of Kummer theory,  $\gamma (K^*)^n$ corresponds to a homomorphism $h:\mathrm{Gal}(K^{(n)}/K) \to \tilde{\mu}_n$
such that the smallest power of $h$ which is in $\mathcal{H}$ is the $n^{th}$
power of $h$.  This means that the compositum $F^{(n)} K(\gamma^{1/n})$ be a cyclic degree $n$
extension of $F^{(n)}$, where $K \subset F^{(n)}$.  Now $K(\gamma^{1/n})$ is an abelian extension of $F$ of some exponent
$m$ with $n| m | n^2$.  If $m \ne n^2$, then $\Gamma = \mathrm{Gal}(K(\gamma^{1/n})/F)$ would have a subgroup $H$ such that $\Gamma/H$ has exponent $n$ and $H$ has exponent $m/n < n$.  
Now $K(\gamma^{1/n})^H$ is an exponent $n$ abelian extension of $F$, so $K(\gamma^{1/n})^H \subset F^{(n)}$. But then $$[F^{(n)}K(\gamma^{1/n}):F^{(n)}] \le [K(\gamma^{1/n}):K(\gamma^{1/n})^H] = H < n$$
contradicting the fact that $F^{(n)} K(\gamma^{1/n})$ is cyclic of degree $n$
over $F^{(n)}$.  Thus $K(\gamma^{1/n})/F$ must be an exponent $n^2$ abelian extension of $F$,
so in fact it is a cyclic extension of degree $n^2$ of $F$.

Suppose now that $x$ is a closed point of $X$ corresponding to a finite place $v$ of $F$.  
As in the proof of Lemma 
\ref{lem:r1comp},  let $O_{x,sh}$ be the strict Henselization of the local ring $O_x$ of $x$ on $X$.  
Define  $F_{x,sh} = F\otimes_{O_F} O_{x,sh}$.   We know that
$Y$ is etale over $x$ so the local ring $O_y$ of each such $y$ lies inside $O_{x,sh}$.
Thus the completion $K_w$ of $K$ at each place $w$ over $v$ lies inside the completion $\hat{F}_{x,sh}$
of $F_{x,sh}$.  Fix a place $w(v)$ of $K$ over $v$ and choose any embedding of $K_{w(v)}$ into
$\hat{F}_{x,sh}$ over $F_v = \hat{F}_x$.  The fact that $G$ permutes the
places $w$ of $K$ over $v$ leads to a sequence of homomorphisms
\begin{equation}
\label{eq:semilocal}
\xymatrix @C=1.2pc{
(K^*/(K^*)^n)^G \ar[r]& (\oplus_{w|v } (K_w^*/(K_w^*)^n))^G  \ar@{=}[d]\\
& (K_{w(v)}^*/(K_{w(v)}^*)^n)^{G_{w(v)}} \ar[r] &  (\hat{F}_{x,sh}^*/(\hat{F}_{x,sh}^*)^n)^{G_x}
}
\end{equation}
when $G_x = \mathrm{Gal}(F_{x,sh}/F_{x,h})$.  However, in Lemma \ref{lem:r1comp} we showed that
the right hand side of (\ref{eq:semilocal}) is just $\hat{F}_x^*/T_x = F_v^*/T_x$. So we can choose the local component
$j_v$ associated to $v$ to come from $\gamma \in K$ in the way described in Lemma \ref{lem:nicegamma}.  
At infinite $v$ we can certainly choose $j_v$ to be trivial.  Now the fact that $\gamma (K^*)^n$
has image $f_X^* c_2$ in $H^2(X,\mu_n)$ together with the  construction of Corollary 
\ref{cor:longer} shows $\omega(z(j)) = f_X^* c_2$.

It remains to show the congruence (\ref{eq:wcongruence}) for each finite place $v$ of $F$.  A uniformizer of $F_v$ is one for $\hat{F}_{x,sh}$ and for $K_w$.  Hence we have 
\begin{equation}
\label{eq:find}
\mathrm{ord}_v(j_v) \equiv \mathrm{ord}_w(\gamma)\quad \mathrm{mod} \quad n \mathbb{Z}
\end{equation} from the above construction of $j_v$ from $\gamma$.  

We now fix a place $w(v)$ over $v$ in $K$.  Since $\gamma (K^*)^n$ lies
in $(K^*/(K^*)^n)^G $ we know that for each place $w$ over $v$ in $K$ there is an integer $t(w)$ such that  $\mathrm{ord}_w(\gamma) = \mathrm{ord}_{w(v)}(\gamma) + n t(w)$.  Since each $K_w$ is an unramified cyclic extension of $F_v$ of degree $n_v$, we have
$$\mathrm{ord}_v(\mathrm{Norm}_{K_w/F_v}(\gamma)) = 
n_v ( \mathrm{ord}_{w(v)}(\gamma) + n t(w)).$$
Hence
\begin{eqnarray}
\label{eq:computit}
\mathrm{ord}_v(\mathrm{Norm}_{K/F}(\gamma)) &=& \sum_{w | v} \mathrm{ord}_v(
\mathrm{Norm}_{K_w/F_v}(\gamma))\nonumber \\
& =& \sum_{w|v} n_v ( \mathrm{ord}_{w(v)}(\gamma) + n t(w))\nonumber \\
& = & (\sum_{w|v} n_v) \cdot \mathrm{ord}_{w(v)}(\gamma) + n \cdot n_v \sum_{w | v} t(w)\nonumber \\
&=& n \cdot \mathrm{ord}_{w(v)}(\gamma)  + n \cdot n_v \sum_{w | v} t(w).
\end{eqnarray}
Since $w(v)$ was an arbitrary place of $K$ over $v$, dividing (\ref{eq:computit})
by $n$ and using (\ref{eq:find}) 
completes the proof of the congruence (\ref{eq:wcongruence}).

\end{proof}

\section{ Proof of Theorem \ref{thm:fixed?}.}
\label{s:proofofmainthm}
\setcounter{equation}{0}

We will adopt the notations of  Theorem  \ref{thm:fixed?}. Thus $G = \mathbb{Z}/n$ and $c_1$ is a generator of $H^1(G,\mathbb{Z}/n) = \mathrm{Hom}(G,\mathbb{Z}/n)$ given by the identity map. The element 
$c_2$ generates $H^2(G,\tilde{\mu}_n)$ and $c = c_1 \cup c_2$.  In Definition \ref{def:roots} we picked a particular primitive $n^{th}$ root of unity $\zeta_n$.  Let $\phi:\mathbb{Z}/n \to \tilde{\mu}_n$ be the isomorphism sending $1$ to $\zeta_n$.
Then  $d_1 = \phi \circ c_1$ is a generator of $H^1(G,\tilde{\mu}_n)$. Write
$K = F(\xi^{1/n})$ as in Corollary \ref{cor:c1analysis} for an element $\xi \in F^*$ which is 
determined mod $(F^*)^n$ by $f_X^* d_1 \in H^1(X,\mu_n) = \mathrm{Hom}(\mathrm{Pic}(X),\tilde{\mu}_n)$.  We have an isomorphism $\mathrm{Gal}(K/F) \to G = \mathbb{Z}/n$ determined by $f^* c_1:\pi_1(X,\eta) \to \mathbb{Z}/n$. We will use this isomorphism to identify $\mathrm{Gal}(K/F)$ with $G = \mathbb{Z}/n$
in what follows.  The element $\xi \in F$ is the Kummer generator for $K$ as an everywhere unramified cyclic extension of $F$ for which 
\begin{equation}
\label{eq:normalization}
\frac{\mathrm{Art}(a)(\xi^{1/n})}{\xi^{1/n}} = \phi(\mathrm{Art}(a)) = f_X^* d_1(a)
\end{equation}
for $a \in Cl(O_F) = \mathrm{Pic}(X)$, where $\mathrm{Art}:Cl(O_F) \to \mathrm{Gal}(K/F)$ is the Artin map for $K/F$.  

The element of $H^2(G,\mathbb{Z}/n)^\star = T/(F^*)^n$
associated to $d_1$ by Corollary \ref{cor:c1analysis} is the coset $\xi (F^*)^n$.
Let $j$ be an idele of $F$ associated to $c_2$ as in Lemma \ref{lem:nicegamma}.  Then
$\omega(z(j)) = f_X^* c_2\in H^2(X,\mu_n)$.  

By Lemma \ref{lem:cupcalc}, the value of the pairing 
$$\langle \ , \ \rangle: H^2(X,\mathbb{Z}/n)^\star \times H^2(X,\mu_n) \to \tilde{\mu}_n$$
in (\ref{eq:h2now}) on the pair $\xi (F^*)^n$ and $\omega(z(j)) = f_X^* c_2$ is 
\begin{equation}
\label{eq:angle}
\langle \xi (F^*)^n,\omega(z(j)) \rangle = \mathrm{Art}(j)(\xi^{1/n})/\xi^{1/n} .
\end{equation}
Here $j$ is the idele constructed in Lemma \ref{lem:cupcalc}, and we are also using $\mathrm{Art}$ to denote the Artin map from the ideles $J(F)$ of $F$ to $\mathrm{Gal}(K/F)$.  
Since $d_1 = \phi \circ c_1$, this is $\phi( n \cdot (f_X^*c_1 \cup f_X^* c_2 ))$  when
$$S(f,c) = f_X^*c_1 \cup f_X^* c_2 \in H^3(X,\mu_ n) =  \mathbb{Z}/n \mathbb{Z}$$
is Kim's invariant for $f$ and $c$.  

Combining this with the normalization of $\xi$ in Corollary \ref{cor:c1analysis} and (\ref{eq:normalization}) gives 
$$\phi( S(f,c)) = \phi(\mathrm{Art}(j)).$$
Thus  $\mathrm{Art}(j) = S(f,c)$.  Hence the proof of the formula
 (\ref{eq:calculation}) is reduced to showing
 \begin{equation}
 \label{eq:arteq}
 \mathrm{Art}(j) = \mathrm{Art}([I])
 \end{equation}
 for a fractional ideal $I$ of $O_F$ having the properties in Theorem \ref{thm:fixed?}, where $[I]$ is the ideal class of $I$ in $Cl(O_F)$.
 
 The first property of $I$ is that it should be an $n^{th}$ root of $\mathrm{Norm}_{K/F}(\gamma) O_F$
 when $\gamma \in K$ is as in Theorem \ref{thm:fixed?}.  The fact that $I$ exists is shown
 by  (\ref{eq:wcongruence}) of Lemma \ref{lem:nicegamma}, which showed $\mathrm{ord}_v(\mathrm{Norm}_{K/F}(\gamma))$ is divisible by $n$ for all finite places $v$ of $F$. Let $j_v$ be the component of $j$ at $v$.  The congruence in (\ref{eq:wcongruence}) also shows that $\mathrm{ord}_v(j_v)$ is congruent to $\mathrm{ord}_v(I)$ modulo the order $n_v$ of the decomposition group of
 a place $w$ over $v$ in $K$. Since $K/F$ is an unramified extension, this is enough to show the
 equality (\ref{eq:arteq}), which completes the proof.
 
 \section{Proof of Theorem \ref{thm:intrinsic} and of Corollaries \ref{cor:easy} and \ref{cor:almost}.}
 \label{s:moreproofs}
 \setcounter{equation}{0}

 The first two parts of Theorem \ref{thm:intrinsic} follow from the arguments used in  Lemmas \ref{lem:boundary}
and \ref{lem:nicegamma} together with Theorem \ref{thm:fixed?}.  

To show the third
part of Theorem \ref{thm:intrinsic}, it will suffice to show the following for each place $v$
of $F$.  Let $j_v$ be the $v$ component of an idele $j$ of $F$ with the properties in Lemma 
\ref{lem:nicegamma}. Let $n_v$ be the local degree of $v$ in $K/F$, i.e. the order of the decomposition group in $G = \mathrm{Gal}(K/F)$ of a place over $v$ in $K$.
In view of the equality (\ref{eq:arteq}), Theorem \ref{thm:fixed?} and the congruence (\ref{eq:wcongruence}), 
it will suffice to show $n_v$ divides $\mathrm{ord}_v(j_v)$ if $L/K$ is unramified over $v$
or if $v$ splits in $K$.  Here $L = K(\gamma^{1/n})$ for some $\gamma$ as in Lemma \ref{lem:nicegamma}.  If $L/K$ is not ramified over the place $w$ of $K$ over $v$, then
$\mathrm{ord}_w(\gamma)$ must be divisible by $n$.  But $\mathrm{ord}_v(j_v) \equiv \mathrm{ord}_w(\gamma)$ mod $n_v\mathbb{Z}$ by (\ref{eq:wcongruence}), and $n_v | n$, so
we get $n_v| \mathrm{ord}_v(j_v)$ in this case.  If $v$ splits in $K$, then $n_v = 1$ so
$n_v | \mathrm{ord}_v(j_v)$ is trivial.  This finishes the proof of Theorem \ref{thm:intrinsic}.

Corollary \ref{cor:easy} follows directly from Theorem \ref{thm:intrinsic}, since we can
take $I'$ to be $O_F$ in this case. 

Suppose now that the hypotheses of Corollary
\ref{cor:almost} hold.  Let $v$ be the place of $F$ determined by the prime 
$\mathcal{P}$ in the statement of Corollary \ref{cor:almost}.  Then there is a unique place $w$ over $v$ in $K$,  $w$ totally
ramifies in $L$, and $v$ and $w$ have residue characteristic prime to $n$.  Thus $L = K(\gamma^{1/n})$ implies $\mathrm{ord}_w(\gamma)$ is relatively prime to $n$, and $n_v = n$ since $v$ is undecomposed in $K$.  Since 
$$\mathrm{ord}_{\mathcal{P}}(I) \equiv \mathrm{ord}_v(j_v) \equiv \mathrm{ord}_w(\gamma) \quad \mathrm{mod}\quad n_v \mathbb{Z}$$
as above, we conclude $\mathrm{ord}_{\mathcal{P}}(I)$ is relatively prime to $n_v = n$.  By part (iii)
of Theorem \ref{thm:intrinsic}, we can take $I' = \mathcal{P}^{\mathrm{ord}_{\mathcal{P}}(I)}$ since $v$ is the only place of $F$ over which $L/K$ ramifies.  Hence $\mathrm{Art}([I'])$ is a generator of $\mathrm{Gal}(K/F)$ since $\mathrm{Art}([\mathcal{P}])$ is,  so Corollary \ref{cor:almost}
follows from (\ref{eq:arteq}).  

 \section{Proof of Theorem \ref{thm:result}.}
 \label{s:yetmoreproofs}
\setcounter{equation}{0}

By assumption, $n > 1$.  It will suffice to construct infinitely many totally complex fields $F$
which have cyclic degree $n^2$ extensions $L_1/F$ and $L_2/F$ having the properties
in Corollary \ref{cor:easy} and \ref{cor:almost}, respectively.  We use a base change argument to do this.

The field $\mathbb{Q}(\zeta_{n^2})$ is totally complex.  We start with an initial choice of a field $F_1$ containing $\mathbb{Q}(\zeta_{n^2})$ together with a cyclic degree $n^2$ extension $N_1/F_1$ of $F_1$.  Let $S_1$ be the set of places of $F_1$ which ramify in $N_1$.  Let $F$ be a number field containing $F_1$ which is linearly disjoint from $N_1$ such that for each place $w$ of $F$ over a place $v$ in $S_1$, the completion $(F_1)_w$ contains the completions of $N_1$ at places over $v$.  Then
$L_1 = F N_1$ will be cyclic unramified degree $n^2$ extension of $F$ as required in Corollary \ref{cor:easy}.  For simplicity we now replace $F_1$ by $F$ to be able to assume that $N_1/F_1$
is a cyclic degree $n^2$ unramified extension.  Any base change of $N_1/F_1$ by a field extension $F$ of $F_1$ which is disjoint from $N_1$ will preserve this property. 

We now focus on finding an extension $F$ of $F_1$ which is disjoint from $N_1$ for which we can construct an extension $L_2/F$ with the properties in Corollary \ref{cor:almost}.

Let $\mathcal{M}$ be a sufficiently high power of the ideal $nO_{F_1}$ in $O_{F_1}$ such that  
if $\alpha \in O_{F_1}$ and $\alpha \equiv 1$ mod $\mathcal{M}$, then $\alpha$ is in  $((F_1)^*_v)^{n^2}$ for all places $v$ of $F_1$ dividing $n$.  Choose a prime $\mathcal{Q}$ of $O_{F_1}$ which splits in the ray class field over ${F_1}$ of conductor $\mathcal{M}$. Then by definition of the ray class group of ${F_1}$ mod $\mathcal{M}$, there is a generator $\alpha$ for $\mathcal{Q}$ such that $\alpha \equiv 1$ mod $\mathcal{M}$.  Since ${F_1}$ contains a root of unity of order $n^2$, the extension $N_2 = {F_1}(\alpha^{1/n^2})$ is an abelian Kummer extension of ${F_1}$.  It is cyclic of degree $n^2$ and totally ramified over $\mathcal{Q}$
since $\alpha$ has valuation $1$ at $\mathcal{Q}$.  Now $N_2/{F_1}$ splits over all places of ${F_1}$ which divide $n$, since by construction $\alpha$ is an $n^2$ power at these places.  Finally, at each place $v$ 
of ${F_1}$ which does not divide $n$ and which is not the place $v_{\mathcal{Q}}$ assocated to $\mathcal{Q}$, $\alpha$ has valuation $0$ at $v$, so $v$ is unramified in $N_2$. Thus $N_2/{F_1}$ is a cyclic degree $n^2$ extension unramified outside of $v_\mathcal{Q}$ and totally ramified over $v_\mathcal{Q}$.

Let $w_{\mathcal{Q}}$ be the unique place over $v_{\mathcal{Q}}$ in $N_2$.  For simplicity, we define
$E$ to be the completion of $F_1$ at $v_{\mathcal{Q}}$, and we let $Y$ be the completion of
$N_2$ at $w_{\mathcal{Q}}$.  Now $Y/E$ is a cyclic degree $n^2$ totally ramified extension of local fields.
There is a unique cyclic unramified extension $E'$ of $E$ of degree $n$. Consider the compositum
$E'Y$.  We have $\mathrm{Gal}(E'Y/E) = \mathrm{Gal}(E'Y/E') \times \mathrm{Gal}(E'Y/Y) = J_1 \times J_2$ where $J_1 = \mathrm{Gal}(E'Y/E') \cong \mathbb{Z}/n^2$ is the inertia subgroup  of
$\mathrm{Gal}(E'Y/E)$ and $J_2 = \mathrm{Gal}(E'Y/Y) \cong \mathrm{Gal}(E'/E) = \mathbb{Z}/n$
is cyclic of order $n$.  Let $j_1$ be a generator of $J_1$ and let $j_2$ be a generator of $J_2$.  
The element $(j_1 , j_2) \in J_1 \times J_2$ then generates a cyclic subgroup $\Gamma$ of order $n^2$ in $J_1 \times J_2$, and $\Gamma \cap (J_1 \times \{0\}) = \Gamma \cap \mathrm{Gal}(E'Y/E')$
has order $n$.  Thus the subfield $E'' = (E'Y)^{\Gamma}$ of $E'Y$ has the property that
$E'Y/E''$ is cyclic of order $n^2$, $\mathrm{Gal}(E'Y/E'') = \Gamma$ has inertia group
$\Gamma \cap (J_1 \times \{0\})$ of order $n$, and $E''/E$ is cyclic and totally ramified of degree $n$.
Thus $E''$ can be obtained from $E$ by adjoining the root of an Eisenstein polynomial of degree $n$
in $O_E[x]$. Note that $E'Y = E'' Y$ since $\Gamma = \mathrm{Gal}(E'Y/E'')$ and $J_2 = \mathrm{Gal}(E'Y/Y)$ intersect only in the identity element.  

We now choose $F$ to be any degree $n$ extension of $F_1$ which is totally ramified
over $v_{\mathcal{Q}}$ such that the completion $F_w$ of $F$ at the unique place $w$ over $v_{\mathcal{Q}}$
is isomorphic to $E''$ as an extension of $E = (F_1)_{v_{\mathcal{Q}}}$.  Such an $F$ 
can be constructed by finding a monic polynomial of degree $n$ in $O_{F_1}[x]$ which is Eisenstein
at $v_{\mathcal{Q}}$ and which locally at $v_{\mathcal{Q}}$ has a root in $E''$.  Because
$F/F_1$ is totally ramified over $v_{\mathcal{Q}}$, it is disjoint from the cyclic unramified degree $n^2$ extension $N_1/F_1$ we constructed at the beginning of the proof.  Hence $FN_1$ is a cyclic degree $n^2$ unramified extension $L_1$ of $F$ of the kind required in Corollary \ref{cor:easy}.  

Consider now
the compositum $L_2 = N_2F$ over $F_1$.     We know there are unique
places $w$ and $w_{\mathcal{Q}}$ over $v_{\mathcal{Q}}$ in $F$ and $N_2$, respectively, and
$F_w = E''$ while $(N_2)_{w_{\mathcal{Q}}} = Y$.  Since $E''Y = E'Y$ has degree $n^3$ over $F_v = E$,
and  $[L_2:F_1] = [L_2:F] \cdot [F:F_1]  \le n^2 \cdot n$, we see $[L_2:F] = n^2$ and there is a unique place $\tilde{w}$ over $v_{\mathcal{Q}}$ in $L_2 = N_2F$.  Thus $L_2/F$ is a cyclic degree $n^2$ extension since it is the base change by $F_1 \subset F$ of $N_2/F_1$. The only place of $F$ which can ramify in $L_2$ is the unique place $w$ over $v_{\mathcal{Q}}$, since $N_2/F_1$ is unramified outside of $v_{\mathcal{Q}}$.  Further, $\tilde{w}$ is the unique place of $L_2$ over $w$, and $(L_2)_{\tilde{w}}/F_w$
is the extension $E''Y/E''$.  We showed that this local extension is cyclic of order $n^2$ with inertia group of order $n$. Thus if we let $\mathcal{P}$ be the prime of $F$ determined by $w$, the extension 
$L_2/F$ will now have all of the properties required in Corollary \ref{cor:almost}.  Theorem \ref{thm:result} now follows from Corollaries \ref{cor:easy} and  \ref{cor:almost}. Note that we can vary the above construction in many ways, e.g. by choosing different primes $\mathcal{Q}$, so we can construct infinitely many $F$ with the properties in Theorem \ref{thm:result}. 

\section{Proof of Theorems \ref{thm:elementary} and \ref{thm:fixitup}}
\label{s:proofalmostlast}
\setcounter{equation}{0}

We will use the notations of Theorems \ref{thm:elementary} and \ref{thm:fixed?}.  Since
$\sigma(y)/y = x^n \in (K^*)^n$, the coset $y(K^*)^n$ lies in $(K^*/(K^*)^n)^G$.  Recall that
we have exact sequences
$$1 \to \tilde{\mu}_n \to K^* \to (K^*)^n \to 1\quad \mathrm{and}\quad 1 \to (K^*)^n \to K^* \to K^*/(K^*)^n \to 1.$$
By the construction of the boundary map $$\delta_0:\hat{H}^0(G,K^*/(K^*)^n)) \to H^1(G,(K^*)^n)$$
the class $\delta_0(y(K^*)^n)$ is represented by the one cocycle which sends $\sigma^i$ to
$\sigma^i(y)/y = (x^n) \sigma(x^n) \cdots \sigma^{i-1}(x^n)$ for $i \ge 0$.  Thus
$\delta_0(y(K^*)^n)$ is the cup product $[x^n] \cup t$, where $[x^n]$ is the class in $H^{-1}(G,(K^*)^n)$
represented by the element $x^n \in K^*$ of norm $1$ to $F$, and  $t$ is an appropriate generator of $H^2(G,\mathbb{Z})$.   
The image of $[x^n]$ under the boundary map $H^{-1}(G,(K^*)^n) \to \hat{H}^0(G,\tilde{\mu}_n) = \tilde{\mu}_n$
is the class represented by $\zeta_n = \mathrm{Norm}_{K/F}(x)$.
Since boundary maps respect cup products with $t$, we find that $\delta_0(y(K^*))$ maps to an element of order $n$ under the boundary map $H^1(G,(K^*)^n) \to H^2(G,\tilde{\mu}_n)$.  This proves
that $y(K^*)^n$ has image of order $n$ under the map $(K^*/(K^*)^n)^G \to H^2(G,\tilde{\mu}_n)$
which was used in (\ref{eq:boundary}) just prior to Theorem \ref{thm:fixed?}.  Hence (\ref{eq:boundary}) shows that if we take  $\gamma = y$ in Theorem \ref{thm:fixed?}, 
Theorem \ref{thm:elementary} now follows from Theorem \ref{thm:fixed?}.  Theorem \ref{thm:fixitup} is proved similarly.

\section{Proof of Theorem \ref{thm:Ralph1}}
\label{s:proofstillnotlast}
\setcounter{equation}{0}

The hypotheses of Theorem \ref{thm:Ralph1} are that  $n$ is a properly irregular prime, so that  $n$ divides $\# Cl(\mathbb{Z}[\zeta_n])$
but not $\# Cl(\mathbb{Z}[\zeta_n + \zeta_n^{-1}])$, and $K$ is a cyclic unramified extension of $F  = \mathbb{Q}(\zeta_n)$ of degree $n$. We
are to show that 
$S(f,c) = 0$ for all surjections $f:\pi_1(X,\eta) \to \mathrm{Gal}(K/F)  = G = \mathbb{Z}/n$
and all $c \in H^3(G,\tilde{\mu}_n)$.

\begin{lemma}
\label{lem:into} There is a $\mathbb{Z}_n$ extension of $F$ which contains $K$ and which is unramified outside $n$.
\end{lemma}

Before proving this lemma, we note how it implies Theorem \ref{thm:Ralph1}.  The lemma shows that there is a cyclic
degree $n^2$ extension $L$ of $F$ which is unramified outside of $n$ and contains $K$.  The unique prime $\mathcal{P}$ over $n$ in $F$ is principal, so it splits in $K$.  Hence Corollary \ref{cor:easy} shows the conclusion of Theorem \ref{thm:Ralph1}  
\medbreak
\noindent \textit{Proof of Lemma \ref{lem:into}}.
Let $E = \mathbb{Q}(\zeta_n + \zeta_n^{-1})$ be the real subfield of $F$. Write $\Gamma = \mathrm{Gal}(F/E) = \{e, \sigma\}$.
Then $\sigma$  acts by inversion on the Sylow $n$-subgroup of $Cl(O_F)$ since $n$ does not divide $\# Cl(O_E)$.  Therefore $K$ is contained in the maximal $n$-elementary extension $N$ of
$F$ which is unramified outside of $p$, Galois over $E$ and for which  $\sigma$ acts by inversion on $\mathrm{Gal}(N/F)$.
 
 The Kummer pairing gives a $\Gamma$-equivariant isomorphism
$\mathrm{Gal}(N/F) = \mathrm{Hom}(T/(F^*)^n,\tilde{\mu}_n)$ when $T = \{\xi \in F^*: \xi^{1/n} \in N\}$.  Since $\sigma$ acts by
inversion on both $\mathrm{Gal}(N/F)$ and $\tilde{\mu}_n$, we conclude that it acts trivially on $T/(F^*)^n$. Because $n$ is odd, 
this implies that the inclusion $E^* \to F^*$ induces a surjection $s:T' \to T/(F^*)^n$ when we let $T' = (E^*)^n \cdot \mathrm{Norm}_{F/E}(T)$.

Since $E$ has class number prime to $n$, and $N/F$ is unramified outside of $n$, we now see that $s(T'') = T/(F^*)^n$ when 
$T''$ is the subgroup of $n$-units in $T'$.   The subgroup of $n$-units in $E$ has no $n$-torsion and rank $d = (n-1)/2$.
We conclude that $\mathrm{Gal}(N/F)$ is an elementary abelian $n$-group of dimension at most $d$ over $\mathbb{Z}/n$.
 
 By class field theory there is a $\mathbb{Z}_n^d$ extension $\tilde{F}$ of $F$ which is unramified outside of $n$ such that $\sigma \in\Gamma$
  acts by inversion on $\mathrm{Gal}(\tilde{F}/F)$.  The maximal $n$-elementary subextension $N'$ of $F$ in $\tilde{F}$ then
has $\mathrm{Gal}(N/F) = (\mathbb{Z}/n)^d$, so in fact $N' = N$.  This implies Lemma \ref{lem:into} since $K \subset N = N'$.

\section{Proof of Theorem \ref{thm:Ralph2}}
\label{s:proofalmostlast2}
\setcounter{equation}{0}

 The hypotheses of Theorem  \ref{thm:Ralph2} are that $n > 2$ is prime and $K/F$ is a cyclic unramified Kummer extension of degree $n$ such that both $K$ and $F$ are Galois over $\mathbb{Q}$.  The action of
 $\Delta = \mathrm{Gal}(F/\mathbb{Q})$ on $G=\mathrm{Gal}(K/F)$ is then via a character 
 $\chi:\Delta \to \mathrm{Aut}(G)$.  If we fix an isomorphism $\iota:G \to  \mathbb{Z}/n$ we get an isomorphism between $\mathrm{Aut}(G)$ and $(\mathbb{Z}/n)^*$ that is independent of the choice of $\iota$.  In this way we can identify $\chi$ with a character $\chi:\Delta \to (\mathbb{Z}/n)^*$.

Theorem \ref{thm:fixed?} gives a $\Delta$-equivariant homomorphism
\begin{equation}
\label{eq:surprise}
H^2(G,\tilde{\mu}_n) = (K^*/(K^*)^n)^G/\mathrm{Image}(F^*) \to G=\mathrm{Gal}(K/F)
\end{equation}
sending the class of $\gamma (K^*)^n$ to $\mathrm{Art}([I])$ in the notation of Theorem \ref{thm:fixed?}.
The action of $\Delta$ on  $G=\mathrm{Gal}(K/F)$ is given by the character $\chi$.
To determine the action of $\Delta$ on $H^2(G,\tilde{\mu}_n)$, we use the exact sequence
$$0 \to \frac{1}{n} \mathbb{Z}/\mathbb{Z} \to \mathbb{Q}/\mathbb{Z} \xrightarrow{\cdot n} \mathbb{Q}/\mathbb{Z} \to 0$$
of modules with trivial $G$-action produced by multiplication by $n$ on
$\mathbb{Q}/\mathbb{Z}$.  The boundary map in the long exact cohomology sequence of this sequence
produces  $\Delta$-equivariant isomorphisms
\begin{equation}
\label{eq:amazin}
 H^2(G,\tilde{\mu}_n) = H^2(G,\mathbb{Z}/n) \otimes \tilde{\mu}_n = 
H^1(G,\mathbb{Q}/\mathbb{Z})\otimes \tilde{\mu}_n.
\end{equation}
Since $H^1(G,\mathbb{Q}/\mathbb{Z})$ is $\mathrm{Hom}(G,\mathbb{Q}/\mathbb{Z})$, we see from (\ref{eq:amazin}) that the action of $\Delta$ on $H^2(G,\tilde{\mu}_n)$ is via the character $\chi^{-1} \cdot \omega$ where $\omega: \Delta \to (\mathbb{Z}/n)^*$ is the Teichm\"uller character giving the action of $\Delta$ on $\tilde{\mu}_n$.  If (\ref{eq:surprise}) is not the trivial homomorphism,
it must be an isomorphism between cyclic groups of order $n$, and since it is $\Delta$-equivariant we would have to have $\chi^{-1}\cdot \omega = \chi$.  This would force $\omega = \chi^2$, which is impossible since
$\omega$ has even order $n-1 = \# (\mathbb{Z}/n)^*$.  Thus (\ref{eq:surprise}) must be trivial under the
hypotheses of Theorem  \ref{thm:Ralph2}, which completes the proof.

\end{document}